\documentclass[a4paper,12pt]{extarticle}

\usepackage[utf8]{inputenc}
\usepackage[T1]{fontenc}
\usepackage{amsmath}
\usepackage{amsfonts}
\usepackage{amssymb}
\usepackage{amsthm}
\usepackage{graphicx} 
\usepackage{cancel}

\usepackage{bm} 
\usepackage{enumerate} 
\usepackage{xcolor} 
\usepackage{stmaryrd}
\usepackage{hyperref}
\usepackage{tikz}
\usetikzlibrary{graphs, graphs.standard}
\usepackage{caption}	
\usepackage{subcaption} 

\graphicspath{{figures/}}  
\usepackage[a4paper, total={7in, 8in}]{geometry}
\allowdisplaybreaks

\newtheorem{thm}{Theorem}[section]
\newtheorem{cor}[thm]{Corollary}
\newtheorem{lem}[thm]{Lemma}
\newtheorem{prop}[thm]{Proposition}

\theoremstyle{definition}
\newtheorem{defn}[thm]{Definition}

\theoremstyle{remark}
\newtheorem{rem}[thm]{Remark}

\newtheorem{exa}[thm]{Example}

\def\AA{\mathcal{A}}

\newcommand{\abs}[1]{\left\vert {#1} \right\vert}

\title{ BMT Independence } 
\date{}

\author{Octavio Arizmendi, Saul Rogelio Mendoza, Josué Vazquez-Becerra}

\begin{document}
\maketitle

\abstract{
We introduce the notion of BMT independence, allowing us to take arbitrary mixtures of boolean, monotone, and tensor independence and generalizing the notion of BM independence of Wysoczanski. 
Pair-wise independence relations are encoded through a directed graph, which in turn determines the way mixed moments must be computed.   
Corresponding Central and Poisson-Type Limit Theorems are provided along with an explicit construction to realize BMT independent random variables as bounded operators on certain Hilbert space. 
}

\tableofcontents

\section{Introduction}

This paper tries to extend some notions in non-commutative probability, where the fundamental framework is a  \emph{non-commutative probability space}, this is, a pair $(\AA,\varphi)$ where $\AA$ is a complex algebra with multiplicative identity $1_\AA$ and $\varphi:\AA\to\mathbb{C}$ is a linear functional so that $\varphi(1_\AA)=1$. 

In this framework, from the probabilistic viewpoint, a fundamental notion is that of independence and the generality of this framework allows considering multiple notions of independence. In order to give a sense of this notion, and in the seek to try to find a way to decide which of these independences one should look at, Ben Ghorbal and Schürmann \cite{GC2002} axiomatized the notion of a \textit{universal products}. For each of these universal products of probability spaces one can associate a notion of independence. They showed that there are three notions of independence satisfying the axioms coming from the universal products. These are Boolean, free and tensor independence. The story was complemented by the work of Muraki \cite{Muraki2003}, when he  generalized the notion of universal product, introducing \emph{natural products}, by removing the \emph{commutativity axiom} (now known as symmetry axiom), which states, in simple words, that for algebras $\mathcal{A}_1$ and  $\mathcal{A}_2$, that fact that $\mathcal{A}_1$ is independent of $\mathcal{A}_2$, implies that $\mathcal{A}_2$ should be independent of $\mathcal{A}_1$ . 

In this new setting, Muraki, showed in \cite{Muraki2003}, that there are only five natural notions of independence, these are, the free, tensor, Boolean, monotone and antimonotone independences. 
In this paper we will be interested in tensor, Boolean and monotone (or antimonotone \footnote{From the fact that $a$ monotone independent to $b$ if and only if $b$ independent antimonotone to $a$, we usually work only with monotone independence (implying the corresponding for the antimonotone independence).
}) cases.

While from the above, one may think of free, Boolean, monotone and tensor independence as parallel theories with no  interaction between them, in this paper we will be interested in mixtures of them. This is, of course, not the first paper that considers such a setting.  In 2007, Wysoczanski \cite{BM2007} introduced BM independence, a generalization of monotone and Boolean independence, giving a framework where some mixtures (following a partial order) of those independences can be represented, giving, in addition, a central limit theorem whose limit is a symmetric distribution that has as particular cases the symmetric Bernoulli and the arcsine distribution. The work \cite{BF2013} gives a framework that combines Boolean and free independences. On a similar vein, the called $\Lambda-$freeness defined in \cite{Lambda2004} mixes free and tensor independence in an algebraic framework. Following the last work, Speicher and Wysoczanski (\cite{eps2016}) represented any mixture of tensor and free variables in terms of a symmetric matrix $\epsilon$ with non-diagonal entries either $0$ or $1$, where $\epsilon_{ij}=0$ represent that algebras $\mathcal{A}_i$ and $\mathcal{A}_j$ are free, and $\epsilon_{ij}=1$ represent that algebras $\mathcal{A}_i$ and $\mathcal{A}_j$ commute and are tensor independent. Finally, the work of   Lenczewski \cite{Len1}, considers  $\Lambda$-Boolean independence and $\Lambda$-monotone independence which mix tensor independence with Boolean independence and with monotone independence, respectively.

In this paper we create a framework that generalizes BM independence, $\Lambda$-Boolean and $\Lambda$-monotone independence, allowing to consider mixtures of tensor, Boolean and monotone independent algebras. This appears to be the first time 3 of the natural independences are combined. We call this new notion of independence BMT (Boolean, monotone and tensor) independence.

To describe such notion we need a graph telling the information about pairwise independence. The idea is  that, vertices correspond to random variables (or algebras), and the relation between two of them is as follows: Vertices joined with a directed edge will correspond to monotone independence between elements, vertices joined by an undirected edge will correspond to tensor independence, and unjoined vertices will correspond to Boolean independence. 
As an example, if we have the variables $\{X_1,X_2,X_3,X_4,X_5\}$ and its independence graph $G$ is the one in the Figure \ref{graph1}, we want the following relations. 

\begin{itemize}
    \item  $X_5$ is Boolean independent from $X_1,X_2,X_3,X_4$
\item $X_4$ is Boolean independent from $X_3$ and $X_1$.
\item $X_2$ is classical independent from $X_4$.
\item  $X_2$ is monotonically independent from $X_3$.
\item $X_1$ monotonically independent from $X_2$ and from $X_3$.
\end{itemize}

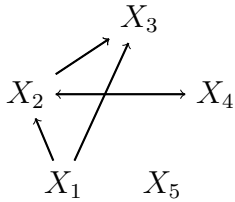
\begin{figure} 
\centering 
  \begin{tikzpicture}
\node (1) at (6.5,.8) {$X_1$};
\node  (2) at (6,2) {$X_2$};
\node (3) at (7.5,3) {$X_3$};
\node(4) at (8.5,2) {$X_4$};
\node(5) at (7.8,.8) {$X_5$};

\path[draw,thick](4) edge node {} (2);
\path[draw,thick](2) edge node {} (3);
\path[draw,thick](1) edge node {} (2);
\path[draw,thick](1) edge node {} (3);
\draw[->]  (1) edge (2) edge (3) ;
\draw[->]  (2) edge (4) ;
\draw[->]  (4) edge (2) ;
\draw[->]  (2) edge (3) ;
\draw[->]  (2) edge (3) ;

 \end{tikzpicture}
  \caption{Independence graph of $(X_1,X_2,X_3,X_4,X_5)$.}
  \label{graph1}
\end{figure}

The main question is, of course, how to define the whole joint distribution, involving not only pairwise relations, and if this joint distribution is representable in a $C^{\ast}$-probability space.  To be clear, BMT variables will have a graph $G$ telling us the independence relations between pairs of variables/algebras, and in the framework we develop, we will have a way for computing  all the mixed moments between these variables.

The main contributions of this paper are, firstly, to give a consistent way to calculate such mixed moment, see Definition \ref{def: BMT independence}, and secondly, to give a concrete construction which satisfies such notion of independence.

Now, for each of the five notions of independence, Boolean, Free, monotone and tensor, there exists a central limit theorem associated with it. Under the tensor independence hypothesis, the limit distribution is the normal distribution; with free independence, the limit is the semicircle distribution; the symmetrical Bernoulli distribution corresponds to Boolean independence: while the arcsine distribution is the limit for the monotone central limit theorem. Here we note that in the usual monotone central limit theorem a total order is imposed in the algebras. In this framework, it is possible to induce monotone and antimonotone independent variables without assuming any specific order. Our third main contribution is proving a central limit theorem for BMT independent variables with a sequence of independence graphs, having as particular cases the monotone, tensor and Boolean central limit theorems.  Similarly, we prove s Poisson limit theorem or law of small numbers.


This work is distributed as follows. In Section 2, we introduce important definitions on graphs and partitions. We proving some new properties of these objects, which will be used crucially in the rest of the paper. In section 3, we introduce the notion of BMT independence with respect to a graph. We provide the first properties and relation with other notions of independence.

In Section 4, we present a model where we deploy the variables whose pairing independence relations are given by a graph $G$.

In Section 5, we present the BMT central limit theorem. We give examples, including Boolean, Free and Monotone CLT's. We also considers some relations between different graphs. Similarly, in Section 6 we consider a Poisson Limit theoerem. We conclude with some open questions and remarks.

%

\section{Set Partitions and Graphs}

%
In this section, we introduce the main combinatorial objects that are used in our analysis of BMT independent random variables. 
First, we describe different types of set partitions. 
Second, we review some terminology on directed graphs. 
Finally, we define the kernel of a function subordinated to a directed graph, which is the object that determines how joint moments of BMT independent random variables must be computed. 

\subsection{Set partitions}

\begin{defn}
A \emph{partition} $\pi = \{ B_1, B_2, \ldots, B_r\}$ of a non-empty set $S$ is a set of non-empty and pair-wise disjoint subsets of $S$ whose union is $S$,
i.e., $B \subset S$ and $B \neq \emptyset$  for every $B \in \pi$, $B\cap B' \neq \emptyset$ implies $B=B'$ for all $B,B'\in \pi$, and $\cup_{B \in \pi } B = S$. 
\end{defn}
We refer to the elements of a partition as \emph{blocks}, and  
the total number of blocks in partition $\pi$ will be denoted by $\#(\pi)$. 
Moreover, a block is said to be \emph{even} if it has even cardinality and \emph{odd} otherwise. 
A partition containing only even blocks is called \emph{even}, but if all of its blocks have exactly two elements we will refer to it as a \emph{pairing}.  

\begin{exa}\label{example_partitions_1}
The sets $\pi_{1} = \{\{1,3\},\{2,4,5,6\}\}$, $\pi_{2} = \{\{1,3,6\},\{2\},\{4,5\}\}$, and $\pi_{3} = \{\{1,3\}  $ $\{4,6\},\{2,5\}\}$ are all partitions of $\{1,2,3,4,5,6\}$. 
The partitions $\pi_{1}$ and $\pi_{3}$ are both even, but while $\pi_{3}$ is a pairing, $\pi_{1}$ is not. 
The partition $\pi_{2}$ is neither even nor odd since it contains two odd blocks, $\{2\}$ and $\{1,6,3\}$, and one even block, $\{4,2\}$. 
\end{exa} 

The set of all partitions of $S$, the set of all even partitions of $S$, and the set of all pairing partitions of $S$ are denoted by $P(S)$, $P_{\text{even}}(S)$, and $P_{2}(S)$, respectively. 
We let $[m]$ and $[\ell,m]$ denote the set of integers $\{1,2,\ldots, m\}$ and $\{\ell,\ell+1,\ldots,m\}$, respectively, for any integers $0 \leq \ell \leq m$. 
When referring to set partitions of $[m]$, we will omit the square brackets. 
For instance, we will write $P_{\text{even}}(m)$ instead of $P_{\text{even}}([ m])$.  

Each partition $\pi \in P(m)$ can be represented graphically by writing the numbers $1,2,\ldots,m$ on a line, drawing vertical lines above each number with matching heights for numbers in the same block, and joining with a horizontal line the vertical lines of the numbers that belong to the same block, see Figure \ref{fig:partition_examples}. 
This leads to the concepts of \emph{nesting} and \emph{crossing} blocks. 

\begin{defn}
Suppose we are given two distinct blocks  $B=\{ b_1 < b_2 < \cdots < b_r\}$ and $C=\{ c_1 < c_2 < \cdots < c_s\}$ from a partition $\pi \in P(m)$. 
We say that $B$ \emph{is nested inside} $C$ if there exists $k \in [1,r-1]$ such that $c_k < b_j < c_{k+1}$ for every $b_j \in B$. 
Additionally, we say $B,C \in \pi$ \emph{cross each other} if there exist $b_i,b_j \in B$ and $c_k,
c_\ell \in C$ such that $b_i < c_k < b_j < c_\ell$. 
\end{defn}

Every partition $\pi \in P(S)$ is equivalent to an equivalence relation $\sim_{\pi}$ on $S$ where $k \sim_{\pi} {k'}$ if and only if $k$ and ${k'}$ belong to the same block of $\pi$. 
Hence, the equivalence relation $\sim_{\pi}$ has the blocks of $\pi$ as equivalence classes.  
The set of partitions $P(S)$ is partially ordered by refinement: we put $\pi \leq \theta$, and say that $\pi$ is a \emph{refinement} of $\theta$, if every block of $\pi$ is contained in some block of $\theta$. 
Notice that $\pi \leq \theta$ if and only if $k \sim_{\pi} {k'}$ implies $k \sim_{\theta} {k'} $ for all $k,{k'} \in S$. 
In Example \ref{example_partitions_1}, the partition $\pi_{3}$ is a refinement of $\pi_{1}$, and there is no other refinement between $\pi_{1}$, $\pi_{2}$, and $\pi_{3}$.  

\begin{figure}
	\centering
	\begin{subfigure}[b]{0.3\textwidth}
		\centering
		\includegraphics[width=\textwidth]{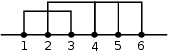}
		\caption{$\pi_1 = \{\{1,3\},\{2,4,5,6\}\}$}
		\label{fig:partition_example_1}
	\end{subfigure}
	\hfill
	\begin{subfigure}[b]{0.3\textwidth}
		\centering
		\includegraphics[width=\textwidth]{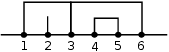}
		\caption{$\pi_2 = \{1,3,6\},\{2\},\{\{4,5\}\}$}
		\label{fig:partition_example_2}
	\end{subfigure}
	\hfill
	\begin{subfigure}[b]{0.3\textwidth}
		\centering
		\includegraphics[width=\textwidth]{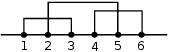}
		\caption{$\pi_3 = \{ \{1,3\},\{2,5\},\{4,6\}\}$}
		\label{fig:partition_example_3}
	\end{subfigure}
	\caption{Graphical representation for partitions $\pi_1$, $\pi_2$, and $\pi_3$ from Example \ref{example_partitions_1}}
	\label{fig:partition_examples}
\end{figure}

\subsection{Digraphs}

\begin{defn}
A \emph{directed graph},  or simply \emph{digraph}, is a pair $G = (V,E)$ where $V$ is a non-empty set, called \emph{vertex set}, and $E$ is a (possibly empty) subset of the Cartesian product $V \times V$, called \emph{edge set}. 
\end{defn}

Given two digraphs $G_1=(V_1,E_2)$ and $G_2=(V_2,E_2)$, we say that $G_1$ is a \emph{subgraph} of $G_2$, and denote this by $G_1 \subset G_2$, if $V_1 \subset V_2$ and $E_1 \subset E_2$.
All digraphs in this paper are assumed to be \emph{simple}, i.e., they contain no \emph{loops} ---edges of the form $(v,v)$. 
The following are some types of digraphs that will be considered in this paper as they concern Boolean, monotone, and tensor independence: 
\begin{itemize}
	\item \emph{Empty digraph.} These are digraphs $G = (V,E)$ without edges, so $E$ is the empty set. 
	\item \emph{Complete digraph.} These are digraphs $G = (V,E)$ containing all possible edges, so every ordered pair $(v,w) \in V \times V$ with $v\neq w$ belongs to $E$. 
	\item \emph{Digraph of a partial order.} These are digraphs $G = (V,E)$ where the vertex set $V$ has a partial order $\preceq$ and the edge set $E$ contains an ordered pair $(v,w) \in V \times V$ if and only if $v \prec w$, i.e., $v \preceq w$ and $v \neq w$. 
\end{itemize}

\subsection{The kernel notation}

Let $S$ and $V$ be non-empty sets. 
Given any function $\bm{i}: S \rightarrow V$,  we make the convention of taking $i_{k} = \bm{i}(k)$ for every $k \in S$.  
Additionally, if  $S = [ m]  $ for some integer $m \geq 1$, then each function  $\bm{i}:  S \rightarrow V$ will be identified with the tuple $(i_{1},i_{2},\ldots,i_{m})$. 

\begin{defn} 
The \emph{kernel} of a function $\bm{i} : S \rightarrow V$ is the partition of $S$ determined by the equivalence relation $k \sim {k'}$ if and only if $i_k = i_{k'}$. 
This partition is denoted by $\ker[\bm{i}]$.
\end{defn}

\begin{rem}
Notice that $\ker[\bm{i}]$ coincides with the partition of $S$ whose blocks are all non-empty pre-images of $\bm{i}$, so we have   
$\ker[\bm{i}] = \{ \{k \in S \mid i_k = v \} \neq \emptyset \mid v \in V  \}$. 
Furthermore, for any partition $\pi \in P(S)$, the condition $\pi \leq \ker[\bm{i}]$ holds if and only if $\bm{i}$ is constant when restricted to each block of $\pi$, i.e., $i_{k} = i_{\ell}$ whenever $k,\ell \in B$ and $B \in \pi$.  
\end{rem}

\begin{exa} 
The function $\bm{i}:[6] \rightarrow [5]$ given by $ (i_{1},i_{2},i_{3},i_{4},i_{5},i_{6}) = (4,1,3,4,1,4)$ has kernel $\ker[\bm{i}] = \{\{1,4,6\},\{2,5\},\{3\}\}$. 
The function $\bm{j}:[6] \rightarrow [5]$ given by $ (j_{1},j_{2},j_{3},j_{4},j_{5},j_{6}) = (5,1,5,4,1,4)$ has kernel $\ker[\bm{i}] = \{\{1,3\},\{2,5\},\{4,6\}\}$. 
\end{exa} 

\begin{defn}\label{kerG} 
The \emph{kernel} of a function $\bm{i} : S \rightarrow V$ \emph{subordinated} to a digraph $G =(V,E)$ is the partition of $S$ determined by the equivalence relation $k \sim k'$ if and only if $i_k = i_{k'}$ 
and $(i_\ell,i_k)$ is an edge of $G$ whenever $i_k \neq i_\ell$ and either $k < \ell < k'$ or $k' < \ell < k$.
We denote this partition by $\ker_G[\bm{i}]$.
\end{defn}

\begin{rem}\label{rmk:kernel_g_extension}
The partition $\ker_G[\bm{i}]$  is a refinement of the partition $\ker[\bm{i}]$. 
Moreover, the second condition in the equivalence relation defining $\ker_G[\bm{i}]$ only concerns $G_{\bm{i}}$ ---the restriction of $G$ to the vertices $i_1,i_2, \ldots, i_m$--- and not the whole graph $G$. 
Hence, if $G_{\bm{i}}$ is the complete graph, the second condition is immediately satisfied and $\ker_G[\bm{i}]$ is defined only by the relation $k \sim k'$ whenever $i_k = i_{k'}$, which is exactly the definition of $\ker[\bm{i}]$, yielding $\ker_G[\bm{i}]=\ker[\bm{i}]$ in this case. 
Therefore, $\ker_G[\bm{i}]$ is not only a refinement but also a generalization of $\ker[\bm i]$. 	
\end{rem}




We have given a sufficient but not necessary condition for $\ker_G[\bm{i}]=\ker[\bm{i}]$ to hold in Remark \ref{rmk:kernel_g_extension}. 
This equality plays an important role in our analysis of BMT random variables, 
so it will be convenient to establish an equivalent condition to it in terms of digraphs. 
To this end, we introduce the following. 

\begin{defn}
The \emph{nesting-crossing graph} of a partition $\pi \in P(m)$ is the digraph $G_\pi=(V_\pi,E_\pi)$ with vertex set $V_\pi = \pi $ and edge set $E_\pi = \left\{ (B,C) \in \pi \times \pi :   B \neq C \text{ and either } C \text{} \right.$ $\left. \text{is nested inside } C \text{ or } B \text{ and } C \text{ cross each other} \right\}$. 
Additionally, given a function $\bm{i}: [m] \to [N]$ with $\ker[\bm{i}] = \pi$, we let $G_{\pi(\bm{i})}$ denote the graph obtained from $G_\pi$ after relabeling each vertex $W \in V_\pi$ as $i_k$ with $k \in W$. 
\end{defn}

\begin{rem}
The edge set $E_\pi$ can be equivalently defined  as the set of all ordered pairs $(B,C) \in \pi \times \pi$ with $B \neq C$ such that there exists $\ell \in B$ with $\min C < \ell < \max C$. 
The latter is also equivalent to the existence of elements $k,k' \in C$ and $\ell \in B$ with $k <\ell < k'$ . 
\end{rem}

\begin{exa}
Consider $\pi=\{ \{1,5\}, \{2,3,7\},\{4\}, \{6\} \}$ and take $B_1=\{1,5\}$, $B_2=\{2,3,7\}$, $B_3=\{4\}$, and $B_4=\{6\}$. 
The nesting-crossing graph $G_\pi$ has then vertex set $V_\pi=\{B_1,B_2,B_3,$ $B_4\}$ and edge set  $E_\pi=\{(B_1,B_2),(B_2,B_1),(B_3,B_1),(B_3,B_2),(B_4,B_2)\}$.
Moreover, the function $\bm{i}=(1,8,8,4,1,5,8)$ satisfies $\ker[\bm{i}]=\pi$, so the graph  $G_{\pi(\bm{i})}$ has vertex set $V_{\pi(\bm{i})}=\{1,4,5,8\}$ and edge set $E_{\pi(\bm{i})}=\{(1,8),(8,1),(4,1),(4,8),(5,8)\}$.
See Figure \ref{fig:partition_graph}.
\end{exa}

\begin{figure}
	\centering
	\includegraphics[scale=1.8]{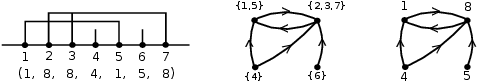}
	\caption{Graphical representation of the partition $\pi=\{ \{1,5\}, \{2,3,7\},\{4\}, \{6\} \}$, the function $\bm{i}=(1,8,8,4,1,5,8)$, the nesting-crossing graph $G_{\pi}$, and the graph $G_{\pi(\bm{i})}$.}
	\label{fig:partition_graph}
\end{figure}


\begin{lem}
 \label{kertopi}
Let $\bm{i}:[m] \to V$ be a function with $\ker[\bm{i}] = \pi$ for some partition $\pi \in P(m)$. 
Then for any digraph $G=(V,E)$ the equality $\ker_G[\bm{i}]=\ker[\bm{i}]$ holds if and only if $G_{\pi(\bm{i})}$ is a subgraph of $G$. 
\end{lem}

Since the vertex set of $G_{\pi(\bm{i})}$ is always a subset of the vertex set of $G$, and the definition of  $\ker_G[\bm{i}]$ and $\ker[\bm{i}]$ implies that the equality $\ker_G[\bm{i}]=\ker[\bm{i}]$ holds if and only if $(i_\ell,i_k)$ is an edge of $G$ whenever $k < \ell < k'$, $i_k \neq i_\ell$, and $i_k = i_{k'}$, the previous lemma is an immediate consequence of the next proposition.

\begin{prop}\label{prop_evaluation_nest_cross}
Suppose $\pi$ is a partition in $P(m)$.
Then, for any function $\bm{i}: [m] \to [N]$ with $\ker[\bm{i}] = \pi$, the edge set $E_{\pi(\bm{i})}$ of the graph $G_{\pi({\bm i})}$ satisfies
\[ E_{\pi(\bm{i})} = \left\{ (i_\ell,i_k)  \mid  \ i_k \neq i_\ell \text{ and there exists } k' \in [m] \text{ with }k < \ell < k', \ i_k = i_{k'} 	\right\}. \]
\end{prop}
\begin{proof}
Put $\widetilde{E} = \left\{ (i_\ell,i_k)  \mid  \ i_k \neq i_\ell \text{ and there exists } k' \in [m] \text{ with }k < \ell < k', \ i_k = i_{k'} 	\right\} $. 
Suppose first $(i_\ell,i_k)$ is an edge in $E_{\pi({\bm i})}$ and let $C$ and $B$ be the blocks of $\pi$ with $ \ell \in C $ and $k \in B$. 
Since $\ker[\bm{i}] = \pi$ and $(i_\ell,i_k) \in E_{\pi({\bm i})}$, we must have $i_k \neq i_\ell$. 
Moreover, the pair $(C,B)$ must be an edge in $E_\pi$, and hence either $C$ is nested inside $B$ or $C$ and $B$ cross each other. 
Now, if $C$ is nested inside $B$, then we can take $k' = \min B$ and $k''= \max B$ to get  $(i_\ell,i_k) = (i_\ell,i_{k'})$,  $i_{k'} \neq i_\ell$, $k' < \ell < k''$, and $i_k = i_{k'}$. 
On the other hand, if $C$ and $B$ cross each other, there must exist $k',k'' \in B$ and $\ell \in C$ with $k' < \ell < k''$; additionally, we get $(i_\ell,i_k) = (i_\ell,i_{k'})$, $i_{k'} \neq i_\ell$, and $i_k = i_{k'}$. 
In any case, we obtain that $(i_\ell,i_k)$ belongs to $\widetilde{E}$. 

Suppose now $(i_\ell,i_k)$ belongs to $\widetilde{E}$ and let $k' \in [m]$ with $k < \ell < k'$ and $i_k = i_{k'}$. 
Take $C$ and $B$ as above. 
Since $\ker[\bm{i}] = \pi$, we have $k'$ belongs to $B$; moreover, $i_k \neq i_\ell$ implies the blocks  $B$ and $C$ are distinct. 
And hence,  either $C$ is nested inside $B$ or $C$ and $B$ cross each other due to $k < \ell < k'$.  
In any case, we obtain $(C,B)$ is an edge in $E_\pi$, and therefore $(i_\ell,i_k)$ belongs to $E_{\pi(\bm{i})}$ since $i_\ell$ and $i_k$ are the replacements of $C$ and $B$, respectively, in the construction of $G_{\pi(\bm{i})}$ from $G_\pi$. 
\end{proof}




\section{BMT Independence}
This is one of the main sections of the paper. 
We first recall some notions from non-commutative probability, in particular, the notions of boolean, monotone, and tensor independence. 
We then introduce the notion of BMT independence, which provides a framework that allows for arbitrary mixtures of boolean, monotone, and tensor independence and, consequently, generalizing the notion of BM independence. 
%
%
%
%
%
%

\subsection{Notions of Independence}
\begin{defn}
	A \emph{non-commutative probability space} is a pair $(\AA,\varphi)$ where $\AA$ is a complex algebra with multiplicative identity $\mathbf{1}_\AA$ and $\varphi:\AA\to\mathbb{C}$ is a linear functional so that $\varphi(\mathbf{1}_\AA)=1$. 
    The elements of $\mathcal{A}$ are called \emph{random variables}. 
\end{defn}
%
%
%
%
%

%
Suppose $a \in \AA$ is a random variable in a non-commutative probability space $(\AA,\varphi)$. 
The value $\varphi(a^k)$ is called the \emph{k-th moment} of $a$. 
Moreover, the sequence $(\varphi(a^k))_{k=1}^\infty$ of all moments of $a$ is called the \emph{(algebraic) distribution} of $a$. 
We say that a probability measure $\mu$ on the real line $\mathbb{R}$ is the   \emph{analytical distribution} of $a$ if for every integer $k\geq 1$ we have 
\[
\varphi(a^k) = \int_{- \infty}^{+\infty} t^k d\mu(t). 
\]
\begin{defn}
    Let $(\AA,\varphi)$ be a non-commutative probability space. 
    An infinite sequence $a_1, a_2, a_3 \ldots $ contained in $\AA$ is said to \emph{converge in moments} to a random variable $a \in \AA$ (resp., $\mu$) if for every integer $k\geq 1$ we have 
    \[
    \lim_{N \to \infty} \varphi(a^k_N) = \varphi(a^k) 
    \quad \left(\text{resp., } =\int_{- \infty}^{+\infty} t^k d\mu(t) \right)  .  
    \]
\end{defn}

For any sequence of random variables  $a_1,a_2,\ldots,a_m \in \mathcal{A}$  and any set $B = \{k_1 < k_2 < \cdots < k_r \} \subset [m] $, we let
\[
(a_k)|_B = a_{k_1} a_{k_2} \cdots a_{k_r} = \prod_{k \in B}^{\rightarrow} a_k .
\]
Suppose $(\mathcal{A}_i)_{i \in I}$ is a family of sub-algebras of $\AA$ and $a_1 \in \AA_{i_1},\ldots, a_m\in \AA_{i_m}$ for some indexes $i_1,\ldots,i_m \in I$. 
We say $a_1 a_2 \cdots a_m$ is \emph{an alternating product of elements of} $(\mathcal{A}_i)_{i \in I}$ if the indexes $i_k$ satisfy that $i_1 \neq i_2$, $i_2 \neq i_3$, $\ldots$, $i_{m-1} \neq i_m$. 

\begin{defn}\label{dfn:boolean_tensor_monotone}
	Let $(\AA,\varphi)$ be a non-commutative probability space. A family $(\mathcal{A}_i)_{i \in I}$  of subalgebras of $\mathcal{A}$ is said to be   
	\begin{enumerate}
		\item[(\textbf{B})] \emph{boolean independent} if for any alternating product $a_1 \cdots a_m$ of elements of  $(\mathcal{A}_i)_{i \in I}$ we have
		\[\varphi(a_1\cdots a_m)=\varphi(a_1)\cdots \varphi(a_m) \] 
		\item[(\textbf{M})] \emph{monotone independent} if $I$ has a linear order $<$ and for any alternating product $a_1 \cdots a_m$ of elements of  $(\mathcal{A}_i)_{i \in I}$ with $a_j \in \mathcal{A}_{i_j}$ we have 
		\begin{enumerate}[M.1]
			\item $\varphi(a_1 \cdots a_m) = \varphi(a_k) \varphi(a_1 \cdots a_{k-1}a_{k+1}\cdots a_n)$ if $ i_{k-1} < i_k >  i_{k+1}$ for some $k \in [2,m-1]$  \ \\ 
			\item $\varphi(a_1 \cdots a_m) = \prod_{\ell=1}^{m}\varphi(a_\ell)$ if $ i_1 > \cdots > i_{k-1} > i_k < i_{k+1} < \cdots < i_m$ for some $k \in [m]$. 
		\end{enumerate}

		\item[(\textbf{T})] \emph{tensor independent} if for any (not necessarily alternating) product $a_1 \cdots a_m$ of elements of  $(\mathcal{A}_i)_{i \in I}$ with $a_j \in \mathcal{A}_{i_j}$ we have 
		\[
		\varphi(a_1\cdots a_m) \ \ = \prod_{B \in \ker({\bm i})} \varphi \left( \prod_{k \in B}^{\to} a_k\right) 
		\]
		where $\prod_{k \in B}^{\rightarrow} a_k := a_{k_1} a_{k_2} \cdots a_{k_r}$ provided $B = \{k_1 < k_2 < \cdots < k_r \}$. 
	\end{enumerate}
\end{defn}

\subsection{BMT independence}
We now present the main definition of this paper, the notion of BMT independence, which relies upon $\ker_G[\bm{i}]$, the kernel of a function subordinated to a digraph $G$, see Definition \ref{kerG}. 
\begin{defn}\label{def: BMT independence}
	Let $(\AA,\varphi)$ be a non-commutative probability space. 
	Suppose $(\mathcal{A}_i)_{i \in I}$ is a family of sub-algebras of $\AA$ and $G=(I,E)$ is a digraph on the set of indices $I$.
	The family $(\mathcal{A}_i)_{i \in I}$ is said to be \emph{BMT independent} with respect to the pair $(\varphi,G)$  if 
	for every integer $m \geq 1$ and variables $a_1 \in \mathcal{A}_{i_1}, a_2 \in \mathcal{A}_{i_2}, \ldots, a_m \in \mathcal{A}_{i_m}$ we have
	\[
	\varphi( a_1 a_2 \cdots a_m ) 
	\ \ =  \prod_{B \in \ker_G[\bm{i}]} 
	\varphi[(a_k)|_B].
	\]
\end{defn}

The above digraph $G=(I,E)$ gives the pair-wise independence relations between sub-algebras from boolean, mononote, and tensor. 
For any two distinct sub-algebras $\AA_i$ and $\AA_j$, we have that 
$\AA_i$ and $\AA_j$ are boolean independent if neither $(i,j)$ nor $(j,i)$ is an edge of $G$, 
$\AA_i$ is monotone independent from $\AA_j$ if $(i,j)$ is an edge of $G$ but $(j,i)$ is not, and 
$\AA_i$ and $\AA_j$ are tensor independent if both $(i,j)$ and $(j,i)$ are edges of $G$. 
Thus, our notion of BMT independence has boolean, mononote, and tensor independence as particular cases. 
This is proved next. 

\begin{prop}\label{prop_BMT_generalizes}
	Let $(\mathcal{A},\varphi)$ be a non-commutative probability space. 
	Suppose $(\mathcal{A}_i)_{i \in I}$ is a family of sub-algebras BMT independent with respect to a digraph $G = (I,E)$. 
	Thus, we have
	\begin{enumerate}[(i)]
		\item the algebras $(\mathcal{A}_i)_{ \in I}$ are tensor independent if $G$ is the graph complete, 
		\item the algebras $(\mathcal{A}_i)_{i \in I}$ are boolean independent if $G$ is the empty graph, and   
		\item the algebras $(\mathcal{A}_i)_{i \in I}$ are monotone independent if $I$ has a total order $<$ and $G$ is the digraph of $<$, i.e., $(i,j)$ is an edge of $G$ if and only if $j < i$. 
	\end{enumerate}
\end{prop}

\begin{proof}
	We will prove \emph{(i)}, \emph{(ii)}, and \emph{(iii)} separately. 
	To this end, let us assume $a_1 \cdots a_m$ is a (not necessarily alternating) product of elements of  $(\mathcal{A}_i)_{i \in I}$ with $a_k \in \mathcal{A}_{i_k}$. 

	\noindent \textbf{Proof of \emph{(i)} .} 
	If $G$ is the complete graph, then $\ker_G[\bm{i}]=\ker[\bm{i}]$ for any $\bm{i} = (i_1, \ldots,i_m)$, see Remark \ref{rmk:kernel_g_extension}. 
	Thus, by Definitions \ref{dfn:boolean_tensor_monotone} and \ref{def: BMT independence}, the algebras $(\mathcal{A}_i)_{ \in I}$ are tensor independent since  
	\[ 
	\varphi(a_1\cdots a_m)
	\ \  = 
	\prod_{B \in \ker_G[\bm{i}]} \varphi((a_k)|_B)		
	\ \  = \ \
	 \prod_{B \in \ker[\bm{i}]} \varphi((a_k)|_B) . 
	\] 

	\noindent \textbf{Proof of \emph{(ii)} .}
	Suppose now $G$ is the empty graph. 
	Note that if $a_1 \cdots a_m$ is an alternating product, then $\ker_G[\bm{i}]=\{\{1\},\{2\},\dots,\{m\}\}$. 
	Indeed, $a_1 \cdots a_m$ is alternating only if $i_1 \neq i_2$, $i_2 \neq i_3$, $\ldots$, $i_{m-1} \neq i_m$, and hence $i_k = i_{k'}$ only if $k=k'$ or $k+1 \leq k'$. 
	But then, since $G$ contains no edges, the second condition defining $\ker_G[\bm{i}]$ is never satisfied when $k+1 \leq k'$, see Definition \ref{kerG}. 
	And therefore, $\ker_G[\bm{i}]$ contains only singletons $\{k\}$ as equivalence classes. 
	Thus, Definitions \ref{dfn:boolean_tensor_monotone} and \ref{def: BMT independence} imply the algebras $(\mathcal{A}_i)_{ \in I}$ are boolean independent since  
	\[ 
	\varphi(a_1\cdots a_m)
	\ \  = 
	\prod_{B \in \ker_G[\bm{i}]} \varphi((a_k)|_B)		
	\ \  = \ \
	\varphi(a_1)\varphi(a_2)\cdots \varphi(a_m). 
	\] 

	\noindent \textbf{Proof of \emph{(iii)} .} 
	Suppose that $\leq$ is a total order of the set $I$ with associated digraph $G=(I,E)$ where $E = \{ (i,j) \in I \times I \mid j < i \}$. 
	Assume $a_1 \cdots a_m$ is an alternating product. 
	We will show that M.1 and M.2 from Definition \ref{dfn:boolean_tensor_monotone} hold. 

	For M.1, we have $ i_{k-1} < i_k >  i_{k+1}$ for some $k \in [2,m-1]$. 
	Then, neither $(i_{k-1},i_k)$ nor $(i_{k+1},i_k)$ is an edge of $G$, and hence $\{k\}$ is a singleton in $\ker_G[\bm{i}]$, see Definition \ref{kerG}. 
	Thus, by BMT independence, we obtain 
	\[
	\varphi( a_1 a_2 \cdots a_m ) 
	\ \ =  \ \ 
			\varphi(a_{k})  
			\prod_{\substack{B \in \ker_G[\bm{i}]\\ B \neq\{k\}}} 
			\varphi((a_{\ell})|_B) .
	\]
	Let $\bm{i}' = (i_1,\ldots,i_{k-1},i_{k+1},\ldots,i_m)$. 
	It is enough to prove that $B \in \ker_G[\bm{i}]$ and $B \neq\{k\}$ if and only if $B \in \ker_G[\bm{i}']$ since BMT independence also gives  
	\[ 
	\varphi(a_1 \cdots a_{k-1}a_{k+1}\cdots a_m)
	\ \ = 
			\prod_{B \in \ker_G[\bm{i}']}
				\varphi((a_{\ell})|_B). 
	\]
	Take $\pi=\ker_G[\bm{i}]$ and $\theta = \ker_G[\bm{i}']$. 
	Due to Definition \ref{kerG}, we need to show that $r\sim_\pi r'$ and $r\sim_\theta r'$ are equivalent for any $r, r' \in  [m] $ with $r,r'  \neq k$. 
	However, $r\sim_\pi r'$ implies $r\sim_\theta r'$ already, so just the other implication is needed. 

	Suppose $r\sim_\theta r'$ with $r<r'$ and $r,r'  \neq k$. 
	Thus, we obtain $i_r=i_{r'}$  and $i_r < i_{\ell}$ whenever $r< \ell <r'$, $i_\ell \neq i_r$, and $\ell \neq k$. 
	To get $r\sim_\pi r'$, the restriction $\ell \neq k$ in the latter condition needs to be lifted.  
	This is immediate if $r' < k$ or $k < r$ or $i_k = i_r$. 
	Thus, without loss of generality, we can assume $r<k<r'$ and $i_k \neq i_r$. 
	Now, if $i_r = i_{k-1}$, then $i_r < i_k$ by hypothesis.  
	On the other hand, if  $i_r \neq i_{k-1}$, we must have $r < k-1 < r'$, and hence $i_r < i_{k-1} < i_k$.
	This shows the restriction $\ell \neq k$ is not needed, and hence $r\sim_\pi r'$. 

	For M.2, we have $i_1 > \cdots > i_{k-1} > i_k < i_{k+1} < \cdots < i_m$ for some $k \in [m]$.  
	Notice $\pi=\ker_G[\bm{i}]$ contains only singletons. 
	Indeed, if $r\neq r'$ and either $1 \leq r,r' \leq k$ or $k \leq r,r' \leq m$, then either $i_r < i_{r'}$ or  $i_{r} < i_{r}$, and hence $r \nsim_\pi r'$ since 	 $i_r \neq i_{r'}$. 
	Thus, for any $r,r' \in [m]$, we have that $r \sim_\pi r'$ only if $r = r'$ or $r < k < r'$ and $i_r = i_{r'}$. 
	However, the latter case implies $i_k < i_r$ and hence $r \nsim_\pi r'$. 
	Therefore, $\ker_G[\bm{i}] = \{\{1\},\{2\},\ldots,\{m\}\}$, and BMT independence gives 
	\[
	\varphi(a_1 a_2 \cdots a_m ) = \varphi(a_1) \varphi(a_2) \cdots \varphi(a_m).
	\]
\end{proof}

\subsection{Weak BM independence}
A main motivation for this paper is the notion of \emph{BM independence}, introduced and investigated by J. Wysoczanski in \cite{BM2007,BM2010} as a generalization of monotone and boolean independence. 
\begin{defn}
	
	Let $(\AA,\varphi)$ be a non-commutative probability space. 
	A family $(\mathcal{A}_i)_{i \in I}$  of subalgebras of $\mathcal{A}$ is said to be \emph{bm-independent} if the set $I$ has a partial order $\preceq$ and  for any alternating product $a_1 \cdots a_n$ of elements of  $(\mathcal{A}_i)_{i \in I}$ with $a_k \in \mathcal{A}_{i_k}$ the following two hold:
	\begin{itemize}
		\item[]  \begin{itemize}
		\item[\textbf{(BM1)}] if  $ i_{k-1} \prec i_k \succ i_{k-1}$ or $ i_{k+1} \nsim i_k \succ i_{k+1}$ or $ i_{k-1} \prec i_k \nsim i_{k+1}$ for some $k \in [2,m-1]$, then 
				\[a_1 \cdots a_m = \varphi(a_k) a_1 \cdots a_{k-1}a_{k+1}\cdots a_m.\] 
		\end{itemize}
		\item[]  \begin{itemize}
		\item[(\textbf{BM2})]  if $i_1\succ \cdots \succ i_k \nsim \cdots \nsim i_{k+\ell}\prec \cdots \prec i_n$ for some $k \in [m]$ and $\ell \in [0,m-k]$, then
				\[ \varphi(a_1 \cdots a_m)  = \prod_{j=1}^m \varphi(a_j). \]
		\end{itemize}
	\end{itemize}
\end{defn}

However, as we will show next, our notion of BMT independence contains a weaker version of BM independence as a particular case. 

\begin{defn} \label{weakBM}
	The family of sub-algebras $(\mathcal{A}_i)_{i \in I}$ from the previous definition is called \emph{weak bm-independent} 
	if any alternating product $a_1 \cdots a_n$ satisfies (\textbf{BM2}) and \textbf{(weak BM1)}, namely, 
	if  $ i_{k-1} \prec i_k \succ i_{k+1}$ or $ i_{k-1} \nsim i_k \succ i_{k+1}$ or $ i_{k-1} \prec i_k \nsim i_{k+1}$ for some $k \in [2,m-1]$, then	
	\[\varphi(a_1 \cdots a_n) = \varphi(a_k) \varphi(a_1 \cdots a_{k-1}a_{k+1}\cdots a_n)\]
\end{defn}

It follows from  the linearity of $\varphi$ that  BM independence implies weak BM independence.

\begin{rem}\label{rmk_digraph_partial}
	Note that if $G=(I,E)$ is the digraph of a partial order $(I,\preceq)$, then $i\prec j$ if and only if $(j,i)\in E$ and $(i,j)\notin E$, and $i\nsim j$ if and only if $(j,i)\notin E$ and $(i,j)\notin E$. 
\end{rem}

\begin{thm}\label{digraphorder} 
	If $G = (I,E)$ is the digraph of a partial order $(I,\preceq)$, then BMT independence and weak BM independence coincide. 
\end{thm}
\begin{proof}
Let $(I,\preceq)$ be a partial order with digraph $G = (I,E)$ where $E = \{ (i,j) \in I \times I \mid j \prec i \}$. 
Suppose $(\AA,\varphi)$ is a non-commutative probability space and $(\mathcal{A}_i)_{i \in I}$ is a family of sub-algebras of $\mathcal{A}$. 
Consider an arbitrary alternating productof $a_1\cdots a_m$ of elements of $(\AA)_{i\in I}$ with $a_k \in \AA_{i_k}$ and take $\bm{i}=(i_1,\ldots, i_m)$. 

Assume first that $(\AA)_{i\in I}$ are BMT independent with respect to $(\varphi,G)$.
Suppose $i_{k-1} \prec i_k \succ i_{k+1}$ or $ i_{k-1} \nsim i_k \succ i_{k+1}$ or $ i_{k-1} \prec i_k \nsim i_{k+1}$. 
To show that weak BM1 holds, one follows a similar argument as in \emph{(iii)} from Proposition \ref{prop_BMT_generalizes} to conclude it is enough to prove that $r\sim_\theta r'$ implies $r\sim_\pi r'$  for any $r, r' \in  [m] $ with $r,r'  \neq k$ where $\pi=\ker_G[\bm{i}]$, $\theta = \ker_G[\bm{i}']$, and $\bm{i}=(i_1,\ldots,i_{k-1},i_{k+1},\ldots,i_m)$. 

Suppose $r\sim_\theta r'$ with $r<r'$ and $r,r'  \neq k$.
Thus, we obtain $i_r=i_{r'}$ and $i_r \prec i_{\ell}$ whenever $r< \ell <r'$, $i_\ell \neq i_r$, and $\ell \neq k$, see Definition \ref{kerG} and Remark \ref{rmk_digraph_partial}. 
Similarly to \emph{(iii)} from Proposition \ref{prop_BMT_generalizes}, we just need to lift the restriction $\ell \neq k$ when $r<k<r'$ and $i_k \neq i_r$. 
Now, if $i_r = i_{k-1}$ or $i_{r'} = i_{k+1}$, then  $i_r \prec i_k$ by hypothesis. 
On the other hand, if  $i_r \neq i_{k-1}$ and $i_{r'} \neq i_{k+1}$, we must have $r < k-1 < k+1 <r'$, and hence $i_r \prec i_{k-1}$ and $i_r \prec i_{k+1}$, yielding $i_r \prec i_k$ by transitivity. 
Thus, weak BM1 is satisfied. 

Suppose $i_1\succ \cdots \succ i_k \nsim \cdots \nsim i_{k+\ell}\prec \cdots \prec i_m$. 
To show that BM2 holds, we will prove $\ker_G[\bm{i}]$ contains only singletons. 
Put $\pi=\ker_G[\bm{i}]$ and take $r,r' \in [m]$ with $r \sim_\pi r'$ and $r \leq r'$. 
Thus, we have $i_r=i_{r'}$ and $i_r \prec i_{\ell}$ whenever $r< \ell <r'$ and $i_\ell \neq i_r$. 
We will show that $r$ and $r'$ must be equal by contradiction. 
Suppose $r < r'$. 
Note that $r+2 \leq r'$ since $a_1  \cdots a_m$ is alternating. 
Now, if $r<k+l$, then $i_{r+1}\neq i_r$ and $r < r+1 < r'$ with $i_r \nprec i_{r+1}$. 
Hence, we obtain $k + \ell \leq r$. 
On the other hand, if $k < r' $, then $i_{r'-1} \neq i_r$ and $r < r'-1 < r'$ with $i_{r'-1} \nsucc  i_{r'}$, yielding $r' \leq k$ and contradicting $r < r'$. 
Therefore, we must have $r = r'$ and so $\ker_G[\bm{i}]=\{\{1\},\ldots,\{m\}\}$. 
It follows from BMT independence that 
$
\varphi(a_1 a_2 \cdots a_m ) = \varphi(a_1) \varphi(a_2) \cdots \varphi(a_m).
$
This proves $(\AA)_{i\in I}$ are weak BM independent. 

Assume now that $(\AA)_{i\in I}$ are weak BM independent. 
Notice that in the first part we actually proved the following two properties: (1) $\ker_G[\bm{i}] = \{\{k\}\} \cup \ker_G[\bm{i}']$ if $i_{k-1} \prec i_k \succ i_{k+1}$ or $ i_{k-1} \nsim i_k \succ i_{k+1}$ or $ i_{k-1} \prec i_k \nsim i_{k+1}$ for some $k \in [2,m-1]$ and (2) $\ker_G[\bm{i}]=\{\{1\},\ldots,\{m\}\}$ if 
$i_1\succ \cdots \succ i_k \nsim \cdots \nsim i_{k+\ell}\prec \cdots \prec i_m$ for some $k \in [m]$ and $\ell\in [0,m-k]$. 
One can verify straightforward that the negation of the hypothesis in (1) is that for every $k \in [2,m-1]$ one has $ i_{k-1} \nsim i_k \nsim i_{k+1}$  or $ i_{k-1} \nsim i_k \prec i_{k+1}$ or  $ i_{k-1} \succ i_k \nsim i_{k+1} $ or $ i_{k-1} \prec i_k \prec i_{k+1}$ or $ i_{k-1} \succ i_k \succ i_{k+1}$. 
It then follows that any alternating sequence $i_1,\ldots, i_m$ satisfies either (1) or (2). 
Therefore, the BMT independence of $(\AA)_{i\in I}$ with respect to $(\varphi,G)$ is obtained from applying induction on the set of all integers $n \geq 1$ such that $ \varphi(a_1 \cdots a_n )  = \prod_{B \in \ker_G[\bm{i}]} \varphi((a_k)|_B)  \text { for any alternating product } a_1 \cdots a_n$ and applying weak BM1 or BM2 accordingly. 
\end{proof}


\subsection{Consistency}
To conclude this section, we show that BMT independence is consistent in the following way.     
If two families of algebras are BMT independent and each element in the first family is tensor (resp., monotone, Boolean) independent from any element in the second family, 
then the whole families are tensor (resp., monotone, Boolean) independent, regardless of the independence relations within each family.

To illustrate this, let us consider algebras $\mathcal{A}_1, \mathcal{A}_2, \mathcal{A}_3$ that are BMT independent with respect to a digraph $G$ containing exactly two edges from $\mathcal{A}_1$ and $\mathcal{A}_2$ to $\mathcal{A}_3$ and possibly more edges connecting $\AA_1$ and $\AA_2$ as follows: 
\begin{center} 
	\begin{tikzpicture}
		\node (1) at (0,2.4) {$\mathcal{A}_1$};
		\node  (2) at (0,.8) {$\mathcal{A}_2$};
		\node (3) at (1.6,1.6) {$\mathcal{A}_3$};

		\draw[dashed]  (1) edge (2);
		\draw[->]  (1) edge (3); 
		\draw[->]  (2) edge (3) ;
	\end{tikzpicture}
\end{center}
Thus $\mathcal{A}_1$ is monotone independent from $\mathcal{A}_3$ and $\mathcal{A}_2$ is also monotone independent from $\mathcal{A}_3$. 
One can then ask what the independence relation between $\mathcal{A}_3$ and $\langle \mathcal{A}_1, \mathcal{A}_2\rangle$, the algebra generated by $\mathcal{A}_1$ and $\mathcal{A}_2$, should be. 
Consistency for BMT independence means that $\mathcal{A}_3$ and $\langle \mathcal{A}_1, \mathcal{A}_2\rangle$ must be monotone independent no matter what the relation between $\mathcal{A}_1$ and $\mathcal{A}_2$ is.

Before we prove that BMT independence is consistent, we need the following two technical lemmas.

\begin{lem}\label{lem_weak_bm}
	%
	Suppose $(\mathcal{A}_i)_{i \in I}$ is a family of sub-algebras BMT independent with respect to a digraph $G = (I,E)$. 
	Let $a_1 a_2 \cdots a_m$ be a product (not necessarily alternating) of elements of  $(\mathcal{A}_i)_{i \in I}$ with $a_k \in \mathcal{A}_{i_k}$ and $\bm{i} = (i_1, i_2,\ldots,i_m)$. 
	If there exist $1 \leq \ell < \ell' \leq m$ so that $\ker_{G}[\bm{i}|_S]$ and $\ker_{G}[\bm{i}|_{S^c}]$ are contained in $\ker_G[\bm{i}]$ with $S = \{ \ell,\ell+1,\cdots,\ell'-1\}$, then 
	$$\varphi(a_1  \cdots a_m) = \varphi(a_{\ell} \cdots a_{\ell'-1}) \varphi(a_1\cdots a_{\ell-1} a_{\ell'} \cdots a_m)$$
\end{lem}
\begin{proof}
	Since $\ker_G[\bm{i}]$, $\ker_G[\bm{i}|_S]$, and $ \ker_G[\bm{i}|_{S^c}]$ are partitions of $[m]$, $S$, and $S^c$, respectively, and $S \cup S^{c} = [m]$, the assumption that both $\ker_G[\bm{i}|_S]$ and $ \ker_G[\bm{i}|_{S^c}]$ are subsets of $\ker_G[\bm{i}]$ implies the partition $\ker_G[\bm{i}]$ is the disjoint union of $ \ker_G[\bm{i}|_S]$ and $\ker_G[\bm{i}|_{S^c}]$. 
	Therefore, the relation $\varphi(a_1 \cdots a_m) = \varphi(a_{\ell} \cdots a_{\ell'-1}) \varphi(a_1\cdots a_{\ell-1} a_{\ell'} \cdots a_m)$ follows directly from the definition of BMT independence. 
\end{proof}

\begin{lem}\label{lem_weak_bm_ver_kernel}
	Let $G = (I,E)$ be a digraph. For any tuple $\bm{i} = (i_1,i_2,\ldots,i_m)$ with $i_k \in I$ and any subset $S \subset [m]$, we have 
	$\ker_{G}[\bm{i}|_S] \subset \ker_G[\bm{i}]$ if and only if $k_1 \sim_{ \ker_G[\bm{i}]} k_2$ for any $k_1,k_2 \in S$ with $k_1 \sim_{ \ker_G[\bm{i}|_S]} k_2$ and $k_1 \nsim_{ \ker_G[\bm{i}]} k_2$ for all $k_1 \in S$ and $k_2 \in [m]\setminus S$. 
\end{lem}
\begin{proof}
	Put $\pi = \ker_G[\bm{i}]$ and $\pi_S = \ker_G[\bm{i}|_S]$. 
	Suppose first $k_1 \sim_{\pi} k_2$ for any $k_1,k_2 \in S$ with  $k_1 \sim_{\pi_S} k_2$ and $k_1 \nsim_{\pi} k_2$ for $k_1 \in S$ and $k_2 \in [m]\setminus S$. 
	Let $V$ be an arbitrary block of $\pi_S$.
	Take any $k \in V$ and let $W$ be the unique block of $\pi$ such that $k \in W$. 
	We will show that $V = W$. 
	Note that $V \subset W$ since for any $k' \in V \subset S$ we have $k \sim_{\pi_S} k'$, and hence $k \sim_{\pi} k'$ by hypothesis.   
	Additionally, we have $W \subset S$ since $k \sim_\pi k''$ for any $k'' \in W$ and by hypothesis $k_1 \nsim_{\pi} k_2$ if $k_1 \in S$ and $k_2 \in [m]\setminus S$. 
	Now, for any $k'' \in W \subset S$, we have $k \sim_\pi k''$, and hence $i_{k} = i_{k''} $ and $(i_{\ell},i_{k}) \in E$ whenever $k < \ell < k''$ if $k \leq k''$, or $k'' < \ell < k$ if $k'' \leq k$, with $\ell  \in [m] \supset S$, so $k \sim_{\pi_S} k''$ and $k'' \in V$.  
	Thus, we get $V = W$. Therefore, since $V$ was arbitrary, we obtain $\ker_{G}[\bm{i}|_S] \subset \ker_G[\bm{i}]$.

	Suppose now $\ker_{G}[\bm{i}|_S] \subset \ker_G[\bm{i}]$. 
	Since  $\pi = \ker_G[\bm{i}]$ and $\pi_S = \ker_G[\bm{i}|_S]$ are the partitions of $[m]$ and $S$ determined by the equivalence relation $k_1 \sim k_2$ only if $i_{k_1} = i_{k_2} $ and $(i_{\ell_1},i_{k_1}) \in E$ whenever $k_1 < \ell_1 < k_2$ with $\ell_1 \in [m]$ and $\ell_1 \in S$, respectively, we have that if $k \nsim_{\pi_S} k'$ for some $k,k' \in S$, then $k \nsim_{\pi} k'$ due to $S$ being a subset of $[m]$.  
	Thus, it only remains to show that $k_1 \nsim_{\pi} k_2$ for any $k_1 \in S$ and $k_2 \in [m]\setminus S$. 
	Take any $k \in S$ and let $V$ the unique block of $\pi_S$  so that $k \in V$.
	By hypothesis, $V$ is also a block of $\pi$, and hence, we obtain $k \nsim_\pi k'$ for any $k' \in [m]\setminus V$, in particular, for any $k' \in [m]\setminus S$. 
\end{proof}

\begin{prop}\label{prop_associativity}
	Let $(\mathcal{A},\varphi)$ be a non-commutative probability space. 
	Suppose $(\mathcal{A}_i)_{i \in I}$ is a family of subalgebras bmt independent with respect to a digraph $G = (I,E)$. 
	If $\{ I_j : j \in J \}$ is a partition of $I$ into non-empty pairwise disjoint subsets and $\mathcal{B}_j = \mathrm{alg}( \mathcal{A}_i : i \in I_j )$, then  
	\begin{enumerate}[(i)]

		\item $(\mathcal{B}_j)_{j \in J}$ are tensor independent if $(i,i') \in E$ whenever $i \in I_{j}$ and $i' \in I_{j'}$ where $j,j' \in J$ with $j \neq j'$  
		
		\item $(\mathcal{B}_j)_{j \in J}$ are boolean independent if $(i,i') \notin E$ whenever $i \in I_{j}$ and $i' \in I_{j'}$ where $j,j' \in J$ with $j \neq j'$

		\item $(\mathcal{B}_j)_{j \in J}$ are monotone independent if $J$ has a total order $<$ so that $(i,i') \in E$ and $(i',i) \notin E$  whenever $i \in I_{j}$ and $i' \in I_{j'}$ where $j,j' \in J$ with $j' < j$

	\end{enumerate}
\end{prop}
\begin{proof}
	Let $b_1 b_2 \cdots b_n$ be alternating product of elements of  $(\mathcal{B}_j)_{j \in J}$ where each $b_r \in \mathcal{B}_{j_r}$ is an alternating product of elements of $(\mathcal{A}_i)_{i \in I_{j_r}}$.   
	Hence, each $b_r$ is of the form  
	\[b_r = a_{(m_0+\cdots+m_{r-1})+1} a_{(m_0+\cdots+m_{r-1})+2}   \cdots  a_{(m_0+\cdots+m_{r-1})+m_r} \]
	with $a_k \in \mathcal{A}_{i_k}$, $i_k \in I_{j_r}$, and $i_{k} \neq i_{k+1}$.   
	Note that we have assumed $m_0=0$.  
	Take $\ell_k = m_0 + m_1 + m_2 + \cdots + m_k$ and $m=\ell_n$. 
	Thus, we have 
	$$\varphi(b_1 b_2 \cdots b_n) \ = \ \varphi( a_1 a_2 \cdots a_{m} ).$$
	%
	%
	%
	%
	
	\noindent \textbf{Proof of \emph{(i)} .} Suppose $(i,i') \in E$ whenever $i \in I_{j}$ and $i' \in I_{j'}$ where $j,j' \in J$ with $j \neq j'$. 
	Put $\bm{j}=(j_1,j_2,\ldots,j_n)$ and $S_U = \cup_{r \in U}[l_{r-1}+1,l_r]$ for each $U \in \ker[\bm{j}] \in P(n)$. 
	%
	%
	We will show that $\pi_U = \ker_G[\bm{i}|_{S_U}]$ is a subset of $\pi = \ker_G[\bm{i}]$ for each $U \in \ker[\bm{j}]$. 

	Take any $U \in \ker[\bm{j}]$.
	Suppose $k_1 \sim_{\pi_U} k_2$ with $k_1 , k_2 \in S_U$ with $k_1 < k_2$ and let $\ell \in [n]$ so that $k_1 < \ell < k_2$ and $i_{\ell} \neq i_{k_1}$.  
	Let $r_1,t \in [n]$ with $r_1 \leq t$ such that $k_1 \in [\ell_{r_1-1}+1,\ell_{r_1}]$ and  $\ell \in [\ell_{t-1}+1,\ell_t]$, so we have $i_{k_1} \in I_{j_r}$ and $i_{\ell} \in I_{j_t}$. 
	Note that $\ell \in S_{U}$ if and only if $j_t = j_{r_1}$. 
	Thus, if $\ell \notin S_U$, then $(i_\ell,i_{k_1}) \in E$ by hypothesis since $j_t \neq j_{r_1}$; on the other hand, if $\ell \in S_U$, then $(i_\ell,i_{k_1}) \in E$ since $k_1 \sim_{\pi_U} k_2$. 
	In any case, we obtain $k_1 \sim_{\pi} k_2$. 
	Suppose now $k_1 \in S_U$ and $k_2 \in [m]\setminus S_U$.
	%
	%
	Let $r_1, r_2 \in [n]$ such that $k_1 \in [l_{r_1-1}+1,l_{r_1}] $ and $k_2 \in [l_{r_2-1}+1,l_{r_2}]$, so we have $i_{k_1} \in I_{j_{r_1}}$ and $i_{k_2} \in I_{j_{r_2}}$. 
	Since $k_2 \notin S_U$, we have $j_{r_2} \neq j_{r_1}$, and hence $i_{k_1} \neq i_{k_2}$ and $k_1 \nsim_{\pi} k_2$ since $I_{j_{r_1}}$ and $I_{j_{r_2}}$ are disjoint. 
	Thus, Lemma \ref{lem_weak_bm_ver_kernel} implies $\ker_G[\bm{i}|_{S_U}]$ is a subset of $\ker_G[\bm{i}]$. 

	Now, since $\ker_G[\bm{i}]$ and $\ker_G[\bm{i}|_{S_U}]$ are partitions of $[m]$ and $S_U$, respectively, and $[m] = \cup_{U \in \ker[\bm{j}]} S_U$, the fact that $\ker_G[\bm{i}|_{S_U}]$ is a subset of $\ker_G[\bm{i}]$ for every $U \in \ker[\bm{j}]$ implies the partition $\ker_G[\bm{i}]$ is the disjoint union of $\ker_G[\bm{i}|_{S_U}]$ with $U \in \ker[\bm{j}]$. 
	It follows from BMT independence that    
	\[
	\varphi( b_1 b_2 \cdots b_n )
	\ \ =  
	\prod_{U \in \ker(\bm{j})} \prod_{ V \in \ker_G(\bm{i}|_{S_U})} \varphi \left( (a_k) | V  \right)
	\ \ =  \prod_{U \in \ker(\bm{j})}  \varphi \left( (b_r) | U  \right). 
	\]
	Therefore,  $(\mathcal{B}_j)_{j \in J}$ are tensor independent. \ \\

	\noindent \textbf{Proof of \emph{ii)}
		.} Suppose now that $(i,i') \notin E$ whenever $i \in I_{j}$ and $i' \in I_{j'}$ with $j,j' \in J$ and $j \neq j'$. 
	Let $(\ell,\ell') = (\ell_0+1,\ell_1+1 )$ and $S = \{ \ell, \ell+1, \ldots, \ell'-1 \}$. 
	Let $\pi$, $\pi_S$, and $\pi_{S^c}$ denote $\ker_G[\bm{i}]$, $\ker_G[\bm{i}|_S]$, and $\ker_G[\bm{i}|_{S^c}]$, respectively.  
	We will show that the partitions $\ker_G[\bm{i}|_S]$ and $\ker_G[\bm{i}|_{S^c}]$ are contained in  $\ker_G[\bm{i}]$. 

	Since $S$ and $S^c$ are sub-intervals of $[m]$, if either $k_1 \sim_{\pi_S} k_2$ with $k_1,k_2 \in S$ or $k_1 \sim_{\pi_{S^c}} k_2$ with $k_1,k_2 \in S^c$, then $k_1 \sim_\pi k_2$.
	Supose now $k_1 \in S$ and $k_2 \in S^c$. 
	Thus, we have $i_{k_1} \in I_{j_1}$ and $i_{k_2} \in I_{j_r}$ for some $r \geq 2$. 
	If $r=2$, then $i_{k_1} \neq i_{k_2}$ since $I_{j_1}$ and $I_{j_2}$ are disjoint due to $j_1 \neq j_2$;
	on the other hand, if $r >3$, then either $i_{k_1} \neq i_{k_2}$ or $i_{k_1} = i_{k_2}$ and $k_1 < \ell' \leq \ell_2 < k_2$ with  $(i_{\ell'},i_{k_1}) \notin E$ since $i_{\ell'} \in I_{j_2}$  and $j_1 \neq j_2$. 
	In any case, we obtain $k_1 \nsim_{\pi} k_2$. 
	It follows from Lemma \ref{lem_weak_bm_ver_kernel} that $\ker_G[\bm{i}|_S]$ and $\ker_G[\bm{i}|_{S^c}]$ are subsets of $\ker_G[\bm{i}]$. 

	Hence, we obtain $\varphi(b_1 b_2 \cdots b_n)  = \varphi(b_1) \varphi(b_2 \cdots b_n)$ by Lemma \ref{lem_weak_bm}.  
	Repeating the same argument for $(\ell,\ell')=(\ell_1+1,\ell_2+1),\ldots,(\ell_{n-2}+1,\ell_{n-1}+1)$, we obtain $\varphi(b_1 b_2 \cdots b_n)  = \varphi(b_1) \varphi(b_2) \cdots \varphi(b_n)$. 
	Therefore, $(\mathcal{B}_j)_{j \in J}$ are boolean independent. \ \\

	\noindent \textbf{Proof of \emph{iii)}
		.} Finally, suppose now that $J$ has a total order $<$ so that $(i,i') \in E$ and $(i',i) \notin E$  whenever $i \in I_{j}$ and $i' \in I_{j'}$ where $j,j' \in J$ with $j' < j$. 
	Assume first $ j_1 > \cdots > j_{r-1} > j_r < j_{r+1} < \cdots < j_n$ for some $r$. 
	The same recursive argument used for boolean independence above can be applied in this case, first for $(\ell,\ell') =(\ell_0+1,\ell_1+1),\ldots,(\ell_{r-2}+1,\ell_{r-1}+1)$, to get $\varphi(b_1 b_2 \cdots b_n)  = \varphi(b_1)  \cdots \varphi(b_{r-1}) \varphi(b_r \cdots b_n)$, and second for $(\ell',\ell) =(\ell_{n-1}+1,\ell_n),\ldots,(\ell_{r-1}+1,\ell_{r})$.
	Note that we go in decreasing order from $n$ to $r$ in the latter case. 
	We then obtain $$\varphi(b_1 b_2 \cdots b_n)  = \varphi(b_1)  \cdots \varphi(b_{r-1}) \varphi(b_r)  \cdots \varphi(b_n).$$ 

	Assume now $j_{r-1} < j_r > j_{r+1}$ for some $r$. 
	Let $\ell = \ell_{r-1}+1$, $\ell' = \ell_r+1$, and $S = \{ \ell, \ell+1, \ldots, \ell'-1 \}$. 
	We will show that $\ker_G[\bm{i}|_S]$ and $\ker_G[\bm{i}|_{S^c}]$ are subsets of $\ker_G[\bm{i}]$. 

	Since $S$ is a sub-interval of $[m]$, we obtain $k_1 \sim_{\pi} k_2$ for any $k_1,k_2 \in S$ with  $k_1 \sim_{\pi_S} k_2$. 
	Suppose now $k_1 \in S$ and $k_2 \in S^c$. 
	Thus, we have $i_{k_1} \in I_{j_r}$ and $i_{k_2} \in I_{j_t}$ for some $ t \in \{1,\ldots,r-1,r+1,\ldots,n\}$.    
	If $t=r+1$, then $i_{k_1} \neq i_{k_2}$ since $I_{j_r}$ and $I_{j_{r+1}}$ are disjoint due to $j_r \neq j_{r+1}$;
	on the other hand, if $t>r+1$, then either $i_{k_1} \neq i_{k_2}$ or $i_{k_1} = i_{k_2}$ and $k_1 < \ell' \leq \ell_{r+1} < k_2$ with  $(i_{\ell'},i_{k_1}) \notin E$ since $i_{k_1} \in I_{j_r}$, $i_{\ell'} \in I_{j_{r+1}}$, and $j_{r+1} < j_{r}$. 
	Similar arguments work if $t \leq r-1$. 
	In any case, we obtain $k_1 \nsim_{\pi} k_2$. 
	It then follows from Lemma \ref{lem_weak_bm_ver_kernel} that $\ker_G[\bm{i}|_S]$ is contained in $\ker_G[\bm{i}]$. 

	To obtain that $\ker_G[\bm{i}|_{S^c}]$ is also contained in $\ker_G[\bm{i}]$, it only remains to prove that $k_1 \sim_\pi k_2$ whenever $k_1 \sim_{\pi_{S^c}} k_2$ with $k_1,k_2 \in S^c$. 
	So, let us assume $k_1,k_2 \in S^c$ with $k_1 \sim_{\pi_{S^c}} k_2$. 
	If $k_1 , k_2 \leq \ell-1$ or $\ell' \leq k_1,k_2$, then $k_1 \sim_\pi k_2$ since $\{1,2,\ldots,\ell-1\}$ and $\{\ell',\ell'+1,\ldots,m\}$ are sub-intervals of $[m]$. 
	Thus, without loss of generality, we can assume $k_1 \leq \ell-l$ and $\ell' \leq k_2$.
	The condition $k_1 \sim_{\pi_{S^c}} k_2$ means $i_{k_1} = i_{k_2}$ and $(i_{\ell''},i_{k_1}) \in E$ for $k_1 < \ell'' < k_2$ with $i_{\ell''} \neq i_{k_1}$ and $\ell'' \in S^c$. 
	Hence, to obtain $k_1 \sim_\pi k_2$, we only need to show that $(i_{\ell''},i_{k_1}) \in E$ for $k_1 < \ell'' < k_2$ with $i_{\ell''} \neq i_{k_1}$ and $\ell'' \in S$. 
	Take $1 \leq t \leq r-1$ so that $i_{k_1} \in I_{j_{t}} $. 
	Now, if $i_{k_1} = i_{\ell-1}$ and $\ell'' \in S$, then $j_t = j_{r-1}$ and $i_{\ell''} \in I_{j_{r}}$, and hence $i_{k_1}\neq i_{\ell''}$ since $I_{j_{r-1}}$ and $I_{j_{r}}$ are disjoint and $(i_{\ell''},i_{k_1}) \in E$ since $j_{r-1} < j_r$.
	On the other hand, if $i_{k_1} \neq i_{\ell-1}$, then $k_1 < \ell -1 < \ell' \leq k_2$ with $\ell - 1 \in S^c$ and $i_{\ell-1} \in I_{j_{r-1}}$, so we must have  $(i_{\ell-1},i_{k_1}) \in E$ since  $k_1 \sim_{\pi_{S^c}} k_2$; however, $(i_{\ell-1},i_{k_1}) \in E$  only if $j_t \leq j_{r-1}$, so we obtain $j_t < j_r$, and hence  $(i_{\ell''},i_{k_1}) \in E$  for any $\ell'' \in S$ since $i_{\ell''} \in I_{j_{r}}$ and $i_{k_1} \in I_{j_t}$. 
	In any case, we get $k_1 \sim_\pi k_2$, and therefore $\ker_G[\bm{i}|_{S^c}]$ is a subset of  $\ker_G[\bm{i}]$ by Lemma \ref{lem_weak_bm_ver_kernel}. 

	It follows from Lemma \ref{lem_weak_bm} that $\varphi(b_1 \cdots b_n) = \varphi(b_r)  \varphi(b_1 \cdots b_{r-1}b_{r+1}\cdots b_n)$. 
	Therefore,  $(\mathcal{B}_j)_{j \in J}$ are monotone independent. 
\end{proof}

\begin{rem}
Associativity of Boolean, monotone, and tensor independence is recaptured by consistency for BMT independence.
To demonstrate this, let us come back to the situation described above regarding  $\mathcal{A}_1, \mathcal{A}_2, \mathcal{A}_3$. 
We add an edge from  $\mathcal{A}_1$ to $\mathcal{A}_2$, so the independence digraph becomes 
\begin{center}
	\begin{tikzpicture}
		\node (1) at (0,2.4) {$\mathcal{A}_1$};
		\node  (2) at (0,.8) {$\mathcal{A}_2$};
		\node (3) at (1.6,1.6) {$\mathcal{A}_3$};
		
		\draw[->]  (1) edge (2);
		\draw[->]  (1) edge (3); 
		\draw[->]  (2) edge (3) ;
	\end{tikzpicture}
\end{center} 
By consistency, Proposition \ref{prop_associativity}, we know that $\mathcal{A}_1$ is monotone independent from $\left< \mathcal{A}_2, \mathcal{A}_3\right>$ and $\mathcal{A}_2$  is monotone independent from $\mathcal{A}_3$. 
Moreover, $\mathcal{A}_1$  is monotone independent from $\mathcal{A}_2$,  
and $\left< \mathcal{A}_1, \mathcal{A}_2\right>$ is monotone independent from $\mathcal{A}_3$. 
Both situations are equivalent due to BMT independence, but this is nothing else than associativity of monotone independence. 
\end{rem}

\section{Construction of BMT algebras of operators }
The purpose of this section is to provide with an analytic framework for BMT random variables. 
Namely, we give an explicit construction to realize any finite family of BMT independent random variables as bounded operators on a Hilbert space. 
Additionally, we show that this construction produces random variables that are BM independent and not just weak BM independent when the corresponding independence digraph comes from a partial order.

Let $H_1,\ldots, H_N$ be complex Hilbert spaces with distinguished unit vectors $\xi_i\in H_i$. 
Let $I_i$ be the identity operator on $H_i$ and let $P_i:H_i\to H_i$ be the orthogonal projection defined by 
\[ P_i(x)=\langle x,\xi_i \rangle_i \xi_i. \]
We consider the non-commutative probability spaces $(B(H_i),\varphi_i)$ where $B(H_i)$ is the space of bounded linear operators on $H_i$ and $\varphi_i:B(H_i)\to \mathbb{C}$ is the vector-state given by
\[\varphi_i(A)=\langle A\xi_i,\xi_i \rangle_i. \]
Let $H$ denote the tensor product of Hilbert spaces $H_1\otimes \cdots \otimes H_N$ with inner product 
\[\langle h_1\otimes \cdots \otimes h_N,h_1'\otimes \cdots \otimes h_N' \rangle=\prod_{i=1}^N \langle h_i,h_i'\rangle_i \]
and let us consider the non-commutative probability space $(B(H),\varphi)$ where $\varphi(T)=\langle T\xi,\xi\rangle$ with unit vector $\xi=\xi_1\otimes \cdots \otimes \xi_N$. 

\begin{defn}
	Given a digraph $G_N=(V_N,E_N)$ with $V_N=[N]$ we define the $\ast$-homomorphism $\pi_i:B(H_i)\to B(H)$ as 
	\[ 
	\pi_i(A)=P_{i,1}\otimes \cdots \otimes P_{i, i-1}\otimes A \otimes P_{i,i+1}\otimes \cdots \otimes P_{i,N}	
	\]
	where $P_{i,j}=I_j$, if $(i,j)\in E_N$, and $P_{i,j}=P_j$, if $(i,j)\notin E_N$.
\end{defn}

\begin{prop}\label{representation_pi}
	The triple $(H,\pi_i,\xi)$ is a representation of $(B(H_i),\varphi_i)$, i.e., $\pi_i: B(H_i)\to B(H)$ is a $\ast$-homomorphism and $\xi$ is a unit vector such that $\varphi_i(A)=\langle \pi_i(A)\xi,\xi\rangle$ for each $A\in B(H_i)$.
\end{prop}
\begin{proof}
	For each $A,B \in B(H_i)$ we have that $\pi_i(AB)=\pi_i(A)\pi_i(B)$ and $\pi_i(A^*)=\pi_i(A)^*$ since 
	$(A_1\otimes \cdots \otimes A_N)(B_1\otimes \cdots \otimes B_N)=(A_1B_1\otimes \cdots \otimes A_NB_N)$, $(A_1\otimes \cdots \otimes A_N)^*=(A_1^*\otimes \cdots \otimes A_N^*)$ and $P_{i,j}=P_{i,j}^2=P_{i,j}^*$. 
	Thus, $\pi_i$ is a $*$-homomorphism. 
	Moreover, since $P_{i,j}\xi_j=\xi_j$ we obtain $\varphi_i(A)=\langle A\xi_i,\xi_i\rangle_i\prod_{j\neq i} \langle P_{i,j}\xi_j,\xi_j \rangle_j=\langle \pi_i(A)\xi,\xi\rangle$.
\end{proof}

\begin{rem}\label{marginalsandhomomorphism}
	Notice that for all $A,B\in B(H_i)$ the projection $P_i$ satisfies
	\begin{align*}
		\langle AP_iB\xi_i, \xi_i \rangle_i=\langle A\xi_i,\xi_i \rangle_i\langle B\xi_i,\xi_i \rangle_i
		\quad \text{and} \quad 
		P_i A P_i=\langle A \xi_i,\xi_i \rangle_i P_i . 
	\end{align*}
	Indeed, since $P_i(x)=\langle x,\xi_i \rangle_i \xi_i$ 
	we obtain $\langle AP_i ( B\xi_i) ,\xi_i\rangle_i = \langle A \langle B\xi_i,\xi_i \rangle \xi_i,\xi_i\rangle_i$ and $[P_i A P_i](x)=\langle x,\xi_i\rangle_i P_i (A\xi_i)=\langle A\xi_i, \xi_i \rangle_i P_i(x)$ for all $x\in H_i$. 
	In terms of the functionals $\varphi_i$, these relations become
	\[
	\varphi_i(AP_iB)=\varphi_i(B)\varphi_i(A)\quad \text{and} \quad P_iAP_i=\varphi_i(A)P_i .
	\] 
\end{rem}

\begin{thm}\label{thm:bmt_operators}
	The family of $\ast$-subalgebras $\{\pi_i(B(H_i))\}_{i=1}^N$ is BMT independent in $(B(H),\varphi)$ with respect to $G_N$.
\end{thm}

\begin{proof}
	Let us denote $\AA_i=\pi_i(B(H_i))$ for $i=1,\ldots, N$. 
	We consider the monomial $a_1a_2\cdots a_n$ with $a_j\in \AA_{i_j}$. 
	Let $A_j\in B(H_{i_j})$ be such that $a_j=\pi_{i_j}(A_j)$ for $j=1,\ldots, n$, this is 
	\[ 
	a_j=P_{i_j,1}\otimes \cdots \otimes P_{i_j, i_j-1}\otimes A_j \otimes P_{i_j,i_j+1}\otimes \cdots \otimes P_{i_j,N}.
	\]
	Put $\bm{i}=(i_1,\ldots, i_n)$ and $m=\#\ker[\bm{i}]$. 
	We also define 
	\[ 
	C=\{c_1<\cdots<c_m\}=\{i_j:1\leq j \leq n\} \quad \text{and} \quad D=\{d_1<\cdots<d_{N-m}\}=[N]\setminus C.
	\]
	The product $a_1a_2\cdots a_n$ can be written as $B_1\otimes \cdots \otimes B_N$ where $B_{d_k}\in \{I_{d_k},P_{d_k}\}$ for $k=1,\ldots, N-m$ and $B_{c_k}$ is of the form
	\[
	Q_{r_1(k)}A_{s_1(k)}Q_{r_2(k)}A_{s_2(k)}\cdots Q_{r_{l(k)}(k)}A_{s_{l(k)}(k)}Q_{r_{l(k)+1}(k)}
	\]
	for $k=1,\ldots,m$ with $\{s_1(k)<\cdots<s_{l(k)}(k)\}=\{j\in [n]: i_j=c_k\}$
	and $Q_{r_w(k)}=I_{c_k}$, if  $(i_l,i_{s_{w}(k)})\in E_N$ for all $s_{w-1}(k)<l<s_{w}(k)$, and $Q_{r_w(k)}=P_{c_k}$, otherwise. 
	Note that $Q_{r_w(k)}\in \{I_{c_k},P_{c_k}\}$ comes from the fact that the variables $a_j$ not in $\AA_{c_k}$ have either $P_{c_k}$ or $I_{c_k}$ in their $c_k$-term in the tensor product that defines them, so between $A_{s_{w-1}(k)}$ and $A_{s_{w}(k)}$ there is a product of $P_{c_k}$'s and $I_{c_k}$'s which is $I_{c_k}$ only if all the elements in between are $I_{c_k}$.
	Then, we have that
	\begin{align*}
		\varphi(a_1a_2\cdots a_n)
		&=\left(\prod_{k=1}^{N-m}\langle B_{d_k} \xi_{d_k},\xi_{d_k}\rangle_{d_k}\right)\left(\prod_{k=1}^{m}\langle B_{c_k} \xi_{c_k},\xi_{c_k}\rangle_{c_k}\right)=\prod_{k=1}^{m}\langle B_{c_k} \xi_{c_k},\xi_{c_k}\rangle_{c_k}
	\end{align*}
	Note that $\langle B_{c_k} \xi_{c_k},\xi_{c_k}\rangle_{c_k}=\langle A_{s_1(k)}Q_{r_2(k)}A_{s_2(k)}\cdots Q_{r_{l(k)}(k)}A_{s_{l(k)}(k)} \xi_{c_k},\xi_{c_k}\rangle_{c_k}$.
	Indeed, if $Q_{r_1(k)}=I_{c_k}$, since $Q_{r_{l(k)+1}(k)} \xi_{c_k}=\xi_{c_k}$ we obtain that 
	\[\langle B_{c_k} \xi_{c_k},\xi_{c_k}\rangle_{c_k}=\langle A_{s_1(k)}Q_{r_2(k)}A_{s_2(k)}\cdots Q_{r_{l(k)}(k)}A_{s_{l(k)}(k)} \xi_{c_k},\xi_{c_k}\rangle_{c_k}; \]
	on the other hand, if $Q_{r_1(k)}=P_{c_k}$, Remark \ref{marginalsandhomomorphism} and the fact that $Q_{r_{l(k)+1}(k)} \xi_{c_k}=P_{c_k}\xi_{c_k}$ imply 
	\[ \langle B_{c_k} \xi_{c_k},\xi_{c_k}\rangle_{c_k}=\langle A_{s_1(k)}Q_{r_2(k)}A_{s_2(k)}\cdots Q_{r_{l(k)}(k)}A_{s_{l(k)}(k)} \xi_{c_k},\xi_{c_k}\rangle_{c_k}.\]
	%
	Let us consider $W_k=\{w_1(k)<\cdots<w_{t(k)}(k)\}=\{w: Q_{r_w(k)}=P_{c_k}\}$ for each $k=1,\ldots,m$. 
	It follows from the definition of $\varphi_{c_k}$ that
	\begin{align*}
		\varphi_{c_k}(B_{c_k})&=\varphi_{c_k}\bigg[\left(\prod_{l=1}^{w_1(k)-1}A_{s_l(k)}\right)P_{c_k}\cdots P_{c_k} \left(\prod_{l=w_{t(k)-1}(k)}^{w_{t(k)}(k)-1}A_{s_l(k)}\right)P_{c_k} \left(\prod_{l=w_{t(k)}(k)}^{l(k)} A_{s_l(k)}\right)\bigg]\\
		&=\varphi_{c_k}\left(\prod_{l=1}^{w_1(k)-1}A_{s_l(k)}\right)
		\cdots \varphi_{c_k} \left(\prod_{l=w_{t(k)-1}(k)}^{w_{t(k)}(k)-1}A_{s_l(k)}\right)\varphi_{c_k} \left(\prod_{l=w_{t(k)}(k)}^{l(k)} A_{s_l(k)}\right)
	\end{align*}
	where the second equality comes from Remark \ref{marginalsandhomomorphism}.
	Now, recall that $\ker_{G_N}[\bm{i}]$ is a refinement of $\ker[\bm{i}]$ where $s_{w-1}(k)\sim s_{w}(k)$ if $(i_l,i_{s_w(k)})\in E_N$ for all $s_{w-1}(k)<l<s_w(k)$. 
	Then, the block $\{s_1(k)<\cdots<s_{l(k)}(k)\}\in \ker[\bm{i}]$ corresponding to the subalgebra $c_k$ is decomposed in $\ker_{G_N}[\bm{i}]$ into $\{s_{1}(k),\ldots,s_{w_1(k)-1}(k)\}$, $\ldots$, $\{s_{w_{t(k)-1}(k)-1}(k),\ldots,s_{w_{t(k)}(k)-1}(k)\}$, $\{s_{w_{t(k)}}(k),\ldots,s_{l(k)}\}$. 
	Thus, we obtain that
	\begin{align*}
		\varphi(a_1a_2\cdots a_n)&=\prod_{k=1}^m \langle B_{c_k}\xi_{c_k},\xi_{c_k}\rangle_{c_k}=\prod_{V\in \ker_{G_N}[\bm{i}]}  \varphi_{c_k}\left(\prod_{s\in V}^{\rightarrow } A_s\right)
	\end{align*}
	where $c_k$ depends on each block $V$ being the common value of all $i_j$ with $j\in V$. 
	For each $V\in \ker_{G_N}[i]$ we have 
	\begin{align*}
		\varphi_{c_k}\left(\prod_{s\in V}^{\rightarrow } A_s\right)=\varphi\left[\pi_{c_k}\left(\prod_{s\in V}^{\rightarrow } A_s\right)\right]=  \varphi\left[\prod_{s\in V}^{\rightarrow } \pi_{c_k}\left(A_s\right)\right]= \varphi\left[\prod_{s\in V}^{\rightarrow } a_s\right]
	\end{align*}
	due to $\varphi_{c_k}=\varphi \circ \pi_{c_k}$, the fact that $\pi_{c_k}$ is an algebra homomorphism, and the definition of the $a_s$. 
	Therefore, we get
	\begin{align*}
		\varphi(a_1a_2\cdots a_n)=\prod_{V\in \ker_{G_N}[\bm{i}]}\varphi\left[\prod_{s\in V}^{\rightarrow } a_s\right]=\prod_{V\in \ker_{G_N}[\bm{i}]} \varphi\left[(a_i)|V\right].
	\end{align*}
	This proves that the subalgebras $\{\pi_j(B(H_j))\}_{1\leq j \leq N}$ are $BMT$-independent in $(B(H),\varphi)$ with respect to $G_N$.
\end{proof}

\begin{rem}
	If we consider the construction of the homomorphisms $\pi_i$ and $G$ is the digraph of a partial order $(V_N,\preceq)$, then we will have $\xi\prec \rho$ if and only if $P_{\rho, \xi}=I_{\xi}$ and $P_{\xi,\rho}=P_\rho$ 
	and $\xi \nsim \rho$ if and only if $P_{\rho,\xi}=P_\xi$ and $P_{\xi,\rho}=P_\rho$. 
	%
\end{rem}

We already prove in the Theorem \ref{digraphorder} that when $G$ is the digraph of a partial order on a finite partially ordered set, BMT subalgebras with respect to $G$ satisfy BM2. However, in general they do not satisfy BM1, but a weaker version of it (the weak BM1 property). However, in the tensor model we have built, we can implement BM independence when considering such $G$. We establish this result in the following theorem giving another construction of a finite collection of $BM$ subalgebras different than the given in \cite{BM2007}.


\begin{thm}
	Suppose $G=([N],E)$ is the digraph of a partial order $\preceq$ on $[N]$. Then the family of $\ast$-subalgebras $\{\pi_{i}(B(H_{i})\}_{i=1}^N$ is BM independent.
\end{thm}

\begin{proof}
	Since the subalgebras $\{\pi_i(B(H_i))\}_{i=1}^N$ are BMT independent, they satisfy weak BM1 and BM2, so we only need to show that BM1 holds. 
	Take $\AA_i=\pi_{i}(B(H_{i}))$ for $i\in [N]$. 
	Assume $a_1\in \AA_\xi$, $a_2\in \AA_\rho$, $a_3\in \AA_\eta$ with $\xi, \rho, \eta \in [N]$ satisfying $\xi\prec \rho \succ \eta$, $\xi \nsim \rho \succ \eta$, or $\xi\prec \rho \nsim \eta$. 
	Recall that each variable $a \in \AA_{i}$ is of the form 
	\[ a = P_{i,1}\otimes \cdots \otimes P_{i, i-1}\otimes T \otimes P_{i,i+1}\otimes \cdots \otimes P_{i,N}=\pi_{i}(T) \]
	for some $T \in B(H_{i})$. 
	So,  let $T_1\in B(H_\xi)$, $T_2\in B(H_\rho)$, and $T_3\in B(H_\eta)$ be such that $a_1=\pi_\xi(T_1)$, $a_2=\pi_\rho(T_2)$ and $a_3=\pi_\eta(T_3)$.
	As in the proof of Theorem \ref{thm:bmt_operators}, we take $B_i,C_i\in B(H_i)$ such that $a_1 a_2 a_3 =B_1\otimes \cdots \otimes B_N$ and $a_1 a_3 =C_1\otimes \cdots \otimes C_N$. 
	To show BM1 holds, it is enough to prove $B_i=C_i$ for $i\neq \rho$ and $B_\rho=\varphi(a_2)C_\rho$. 
	%
	%
	%

	
	Observe that $B_j=P_{\xi,j}P_{\rho,j}P_{\eta,j}$ and $C_j=P_{\xi,j}P_{\eta,j}$ for $j\ \neq \xi,\rho,\eta$. 
	If $P_{\xi,j}=P_j$ or $P_{\eta,j}=P_j$, then $B_j = C_j = P_j$  because each operator $P_{\xi,j}$, $P_{\rho,j}$, or $P_{\eta,j}$ is either the projection $P_j$ or the identity $I_j$. 
	%
	On the other hand, if $P_{\xi,j}=P_{\eta,j}=I_j$, due to the definition of the operators $P_{i,j}$ and since $G$ is the digraph associated to $\preceq$, we  have $j\prec \eta$ and $j\prec \xi$. 
	But, by hypothesis, either $\xi \prec \rho$ or $\eta \prec \rho$ holds, and hence $j\prec \rho$ and $P_{\rho,j}=I_j$. 
	Therefore, $B_j=C_j$ for $j \neq \xi, \rho, \eta$. 

	To prove $B_j=C_j$ for $j = \xi, \eta$ and $B_\rho=\varphi(a_2)C_\rho$, we consider two cases:  $\xi=\eta$ and $\xi\neq \eta$. 
	Suppose first  $\xi=\eta$. 
	Thus, none of the two conditions $\xi \nsim \rho \succ \eta$ or $\xi\prec \rho \nsim \eta$ holds, so we must have $\xi\prec \rho$, and hence $P_{\rho,\xi}=I_\xi$ and $P_{\xi,\rho}=P_\rho$. 
	Moreover,  $B_\xi=T_1 P_{\rho,\xi}T_3$, $C_\xi=T_1 T_3$, $B_\rho=P_{\xi,\rho}T_2 P_{\xi,\rho}$,  and $C_\rho=P_{\xi,\rho}P_{\xi,\rho}$ since  $a_1=\pi_\xi(T_1)$, $a_2=\pi_\rho(T_2)$, and $a_3=\pi_\eta(T_3)$. 
	Hence, $B_\xi = C_\xi$ due to $P_{\rho,\xi} = I_\xi$ and $B_\rho= P_\rho T_2 P_\rho = \varphi(a_2)C_\rho$ due to $P_{\xi,\rho} = I_\rho$ and Remark  \ref{marginalsandhomomorphism}.  

	Suppose now $\xi \neq \eta$. 
	In this case, we have that $B_\xi=T_1 P_{\rho,\xi}P_{\eta,\xi}$,  $C_\xi=T_1 P_{\eta,\xi}$,  $B_{\eta}=P_{\xi,\eta}P_{\rho,\eta}T_3$, $C_\eta=P_{\xi,\eta}T_3$, $B_\rho=P_{\xi,\rho}T_{2}P_{\eta,\rho}$, and $C_\rho=P_{\xi,\rho}P_{\eta,\rho}$.  
	Since $P_{i,j}=P_j$ if $i\nsim j$ or $i\prec j$, then $P_{\xi,\rho}=P_{\rho}=P_{\eta,\rho}$ provided $\xi \prec \rho \succ \eta$, $\xi \prec \rho \nsim \eta$ or $\xi \nsim \rho \succ \eta$. 
	Thus, Remark \ref{marginalsandhomomorphism} gives $B_\rho=P_\rho T_2 P_\rho=\varphi_{\rho}(T_2) P_\rho=\varphi(a_2)C_\rho$.  
	For $B_\xi=C_\xi$ and $B_\eta=C_{\eta}$, let us consider three sub-cases: $\xi \nsim \eta$, $\xi \prec \eta$ and $\xi \succ \eta$.  
	First, if $\xi \nsim \eta$,  so we get $P_{\eta,\xi}=P_\xi$ and $P_{\xi,\eta}=P_\eta$, and hence $B_\xi=T_1 P_\xi=C_\xi$ and $B_\eta=P_\eta T_3=C_\eta$ since  each $P_{\rho,\xi}$ and $P_{\rho,\eta}$ is either an identity $I_j$ or a projection $P_j$. 
	Second, if $\xi \prec \eta$,  so we get $P_{\eta,\xi}=I_\xi$ and $P_{\xi,\eta}=P_\eta$, and hence $B_\xi=T_1 P_{\rho,\xi}$ and $C_\xi=T_1$; moreover, since $P_{\rho,\eta}$ is either the identity $I_\eta$ or the projection $P_\eta$, we obtain $B_\eta=P_\eta T_3=C_\eta$. 
	Note that if $\eta \prec \rho$, then $\xi \prec \eta \prec  \rho$, and hence  $P_{\rho,\xi}=I_\xi$ and $B_\xi=C_\xi$;  
	on the other hand, if $\eta \nsim \rho$, then we must have $\rho \succ \xi$, and hence $P_{\rho,\xi}=I_\xi$ and $B_\xi=C_\xi$. 
	Finally, if $\xi\succ \eta$, follwing similar arguments to the case $\xi\prec \eta$, we get $B_\xi=T_1 P_\xi =C_\xi$ since $P_{\eta,\xi}=P_\xi$  and $B_{\eta}=P_{\xi,\eta} T_3 = C_\eta $ since $\rho \succ \eta$ and $P_{\rho,\eta} = I_\eta$. 
\end{proof}

Let us mention that the idea of using tensor products and rank 1 projection to construct independent algebras of operators is not new. 
 Lenczewski \cite{Len2019} gave a tensor model for Boolean independent random variables, which was extended by Franz to include monotone and anti-monotone notions of independence in  \cite{Franz}.

More interestingly and related to our construction,  Lenczewski  gave a tensor model of $\Lambda$-boolean (mixtures of Boolean and tensor) independence and $\Lambda$-monotone (mixtures of monotone and tensor) independence, \cite{Len1}.  More recently he also gave a construction in \cite{Len2019} for $c$-monotone independence \cite{Hase2011}.

\section{BMT Central Limit Theorem}
In this section, we prove the Central Limit Theorem (CLT) for BMT independent random variables  together with some of its properties regarding the possible limiting distributions. 
In particular, we recover the known CLTs for Boolean, monotone, and tensor independence and give sufficient conditions for the non-compactess of the limiting measure. 

Through the entire section, it is assumed that we are given a non-conmutative probability space $(\mathcal{A},\varphi)$ together with a sequence of random variables $(a_i)_{i=1}^\infty$ that are identically distributed with zero mean and unit variance
\footnote{As it is customary in non-commutative probability, the identically distributed assumption can be relaxed to having uniformly bounded moments, i.e., $\sup_{i} |\varphi(a_i^n)| < \infty$ for each integer $n\geq 1$.)}.
Each finite sequence $a_1, a_2, \ldots, a_N$ is assumed to be BMT independent with respect to a digraph $G_N = (V_N,E_N)$ with $V_N = [N]$ and $G_{N-1} \subset G_{N}$. 

\subsection{Central Limit Theorem}

The Central Limit Theorem for BMT independent random variables refers then to determining the limiting distribution of the normalized sum $(a_1+\cdots+a_N)/\sqrt{N}$. 
Concretely, for each integer $k \geq 1$, and up to an error term of order $N^{-1/2}$, we compute the value of 
\[ \varphi \left[ \left( \frac{a_1+\cdots+a_N}{\sqrt{N}}\right)^k \right] .\]

A first step is to single out the monomial containing singletons, as in most of the combinatorial proofs of the CLT.  

\begin{prop}[Singleton condition]\label{prop_singleton_cond} 
	Let $a_1,\ldots,a_N$ be centered BMT independent random variables. Let  $\bm{i}: [m] \to [N]$  be a set of indices. If $\pi = \ker[\bm{i}]$ contains a singleton, then 
$\varphi(a_{i_1}\cdots a_{i_m})=0.$
\end{prop}
\begin{proof}
Notice that this follows directly from the definition of BMT independence and the fact that $\varphi(a_i) = 0$ for any $i$. 
Indeed, if $\{k\}$ is a singleton in $\ker[\bm{i}]$, so is in $\ker_G[\bm{i}]$, which is a refinement of the former, and thus 
\[
\varphi(a_{i_1}\cdots a_{i_m})=	\varphi(a_{i_k})\prod_{\substack{V \in \ker_G[ \bm{i} ]\\V\neq\{k\} }} \varphi \bigg( \prod_{k \in V}^{\rightarrow} a_{i_k} \bigg)=0 .
\]
\end{proof}

\begin{prop}\label{prop_cerouno_ver2}
	%
	%
	Suppose  $\bm{i}: [2m] \to [N]$ is such that $\pi = \ker[\bm{i}]$ is pair partition in $\mathcal{P}_2(2m)$ and let $G_{\bm{i}}$ denote the independence graph of $a_{i_1},\ldots, a_{i_{2m}}$. 
	Thus, we have 
	\[
	\varphi(a_{i_1}\cdots a_{i_m})=\begin{cases}
		1,\text{ if } G_{\pi(\bm{i})} \subseteq  G_{\bm{i}}, \\
		0 ,\text{ otherwise.} 
	\end{cases}
	\] 
	%
\end{prop}
\begin{proof}
	Note first that the condition $ G_{\pi(\bm{i})} \subseteq  G_{\bm{i}}$ is equivalent to $E_{\pi(\bm{i})} \subseteq  E_{\bm{i}}$ since the graphs $G_{\pi(\bm{i})}$ and $ G_{\bm{i}}$ have the same set of vertices. 
	Now, the BMT independence of $a_1,\ldots,a_N$ implies 
	\begin{eqnarray*}
		\varphi(a_{i_1}\cdots a_{i_m})=
		\prod_{V \in \ker_G[ \bm{i} ] } 
		\varphi \bigg( \prod_{k \in V}^{\rightarrow} a_{i_k} \bigg)  
	\end{eqnarray*}
	where $\ker_G[ \bm{i} ]$ is a refinement of $\ker[\bm{i}]$. 
Now, if $\ker_G[ \bm{i} ]\neq \ker[\bm{i}]$,    then $\ker_G[ \bm{i}]$ has a singleton since $\ker[\bm{i}]$ is a pair partition. This in turns implies by Proposition \ref{prop_singleton_cond}, that $\varphi(a_{i_k}) = 0$, since $a_{i_k}$ are assumed to be centered. On the other hand,  if $\ker_G[ \bm{i}]=\ker[\bm{i}]$, we obtain $\varphi(a_{i_1}\cdots a_{i_m}) = \prod_{V \in \ker_G[ \bm{i} ] } 
	\varphi ( \prod_{k \in V}^{\rightarrow} a_{i_k} ) =1$ since $\varphi(a^2_{i_k}) =1$ for each $k$. 
	Finally, the condition $\ker_G[ \bm{i}]=\ker[\bm{i}]$ is equivalent to the condition  $G_{\pi(\bm{i})} \subseteq  G_{\bm{i}}$ by Lemma \ref{kertopi}.
\end{proof}

\begin{thm}[BMT central limit theorem]\label{thm_bmt_clt}
	Suppose $a_1,\ldots,a_N$ are centered variables with unit variance, uniformly bounded  moments of all orders, and independence graph $G_N$. 
	Then for any integer $m\geq 1$ we have 
	\begin{align}\label{momentsCLT}
		\varphi \left(\frac{a_1 + \cdots + a_N}{\sqrt{N}}\right)^m
		=
		\sum_{\pi \in P_2(m)}  \ \ 
		N^{-m/2}
		\sum_{\substack{ {\bm i}:[m]\rightarrow [N] \\ \ker({\bm i})=\pi}}  
		\bm{1}_{G_{\pi(\bm{i})}  \subseteq G_{\bm i}}  
		\quad	+ \quad O(N^{-1/2}) \ .  
	\end{align}
	where $G_\pi$ is the nesting-crossing graph of $\pi$ and $G_{{\bm i}}$ is independence graph fo $a_{i_1},a_{i_2},\ldots,a_{i_{2k}}$ for ${\bm i} = (i_1,i_2,\ldots,i_{m})$. 
\end{thm}
\begin{proof}
	Observe that
	$$\left( \frac{a_1+\cdots+a_N}{\sqrt{N}}\right)^m=\frac{1}{N^{m/2}}\sum\limits_{i_1=1}^{N}\cdots \sum\limits_{i_m=1}^{N} \left(\prod_{k=1}^m a_{i_k}\right).$$
	From the fact that $\bm{i}\sim \bm{j}$ if $\ker[\bm{i}]=\ker[\bm{j}]$ defines an equivalence relation, we can split the set of functions $\bm{i}:[m]\to [N]$ into disjoint sets. Using this and the linearity of $\varphi$, we have that
	$$\varphi\left[\left( \frac{a_1+\cdots+a_N}{\sqrt{N}}\right)^m\right]=\frac{1}{N^{m/2}}\sum\limits_{\pi \in \mathcal{P}(m)} \sum\limits_{\substack{\bm{i}:[m]\to [N] \\ \ker[i]=\pi}} \varphi(a_{i_1}\cdots a_{i_m}).$$
	The variables $(a_i)$ are BMT independent, so we get 
	$$\varphi\left[\left( \frac{a_1+\cdots+a_N}{\sqrt{N}}\right)^m\right]=\frac{1}{N^{m/2}}\sum\limits_{\pi \in \mathcal{P}(m)} \sum\limits_{\substack{\bm{i}:[m]\to [N] \\ \ker[i]=\pi}} \prod_{V\in \ker_G[i]} \varphi((a_{i_k})| k\in V).$$
	From Proposition \ref{prop_singleton_cond}, if $\pi$ contains a singleton, then $\prod_{V\in \ker_G[i]} \varphi((a_{i_k})| k\in V)=0$. 
	So, only partitions $\pi$ with at least two elements per block contribute in the sum above. 
	Take $C>0$ such that $C\geq \sup_{n\leq m} |\varphi(a_i^n)|$. Since $i_k=i_{k'}$ for any $k,k' \in V$ and $V \in \ker_G[\bm{i}]$, we have that
	$$\prod_{V\in \ker_G[\bm{i}]} | \, \varphi((a_{i_k})|k\in V) \, |\leq C^{|\ker_G[i]|}\leq C^m.$$
	Thus, for any partition $\pi\in \mathcal{P}(m)$, we obtain
	\begin{align*}
		\abs{\sum\limits_{\substack{i:[m]\to [N] \\ \ker[\bm{i}]=\pi}} \prod_{V\in \ker_G[i]} \varphi((a_{i_k})| k\in V) } &\leq C^m\sum\limits_{\substack{\bm{i}:[m]\to [N] \\ \ker[\bm{i}]=\pi}} 1\leq C^m N^{|\pi|}.
	\end{align*}
	This implies 
	\begin{align*}
		N^{-m/2} \sum\limits_{\substack{\pi\in \mathcal{P}(m) \\ |\pi|<m/2}}
		\abs{\sum\limits_{\substack{i:[m]\to [N] \\ \ker[\bm{i}]=\pi}} \prod_{V\in \ker_G[i]} \varphi((a_{i_k})| k\in V) } \leq \sum\limits_{\substack{\pi\in \mathcal{P}_\chi (m) \\ |\pi|<m/2}} C^m N^{|\pi|-m/2}=O(N^{-1/2}).
	\end{align*}
 	Let us denote by $\tilde{\mathcal{P}}(m)$ the set of all partitions $\pi \in \mathcal{P}(m)$ with no singletons.  We have proved that only partitions in  $\tilde{\mathcal{P}}(m)$ give a non-zero contribution and if $|\pi|< m/2$ this contribution is of order  $O(N^{-1/2})$.  Thus noticing that  $|\pi|\leq m/2$ for any $\pi \in \tilde{\mathcal{P}}(m)$ with equality only if $\pi$ is a pair partition we arrive to 
  	\begin{align*}
		\varphi \left(\frac{a_1 + \cdots + a_N}{\sqrt{N}}\right)^m
		=
		\sum_{\pi \in P_2(m)}  \ \ 
		N^{-m/2}
		\sum_{\substack{ {\bm i}:[m]\rightarrow [N] \\ \ker({\bm i})=\pi}}  
			\varphi(a_{i_1}\cdots a_{i_m})	+ \quad O(N^{-1/2}) \ .  
	\end{align*}
Finally, Proposition  \ref{prop_cerouno_ver2} gives 
	\begin{align*}
		\varphi \left(\frac{a_1 + \cdots + a_N}{\sqrt{N}}\right)^m
		=
		\sum_{\pi \in P_2(m)}  \ \ 
		N^{-m/2}
		\sum_{\substack{ {\bm i}:[m]\rightarrow [N] \\ \ker({\bm i})=\pi}}  
		\bm{1}_{G_{\pi(\bm{i})}  \subseteq G_{\bm i}}  
		\quad	+ \quad O(N^{-1/2}) \ .  
	\end{align*}
\end{proof}

\begin{rem} \label{oddmoments}
The set of pairing partitions $P_2(m)$ is empty if $m$ is odd. 
Thus, Theorem \ref{thm_bmt_clt} states that odd moments in the CLT for BMT random variables always vanish as $N\to \infty$, and therefore the limiting distribution must be symmetric if it exists. 

Moreover, the even moments satisfy the Carleman's condition, namely, 
$\sum^\infty_{k=1}m_{2k}^{-1/2k}=+\infty$
where $m_{2k}$ denotes the $2k$-th moment since Theorem \ref{thm_bmt_clt} gives that each $m_{2k}$ is bounded by $|\#P_2(2k)|=1\cdot 3\cdots 
(2k-1)$. 
Consequently, the limiting distribution, if it exists, is determined by moments, and convergence in distribution for BMT central limit theorem is equivalent to convergence in moments.
\end{rem}


\begin{cor}[Boolean Central Limit theorem]\label{cor_boolean_ctl}
	If the independence graph $G_N$ of the variables $a_1,a_2,\ldots,a_N$ is the null graph for every integer $N\geq 1$, then $(a_1 + \cdots + a_N)/\sqrt{N}$ converges in moments as $N \to \infty$ to the Bernoulli distribution $(\delta_{-1} + \delta_{+1})/2$.  
\end{cor}

\begin{proof}
	By Remark \ref{oddmoments} is is enough to consider even moments, so let $m=2k$. For all $\bm{i}:[m]\to [N]$, $G_{\mathbf{i}}$ is a graph with no edges. So,  $\mathbf{1}_{G_{\pi(\bm{i})} \subseteq G_{\bm{i}}}$ is $1$ only  when $G_{\pi(\bm{i})}$ has no edges, i.e. only when $\pi=\{\{1,2\},\ldots, \{2k-1,2k\}\}$. Thus,
	\begin{align*}
		\lim_{N\to \infty}\varphi\left( \left(\frac{a_1+\cdots+a_N}{\sqrt{N}} \right)^{2k} \right)&=\lim_{N\to \infty}\frac{1}{N^{k}} \sum\limits_{\underset{\ker[i]=\{\{1,2\},\ldots, \{2k-1,2k\}\}}{\bm{i}: [m] \to [N]}} 1\\
		&=\lim_{N\to \infty}\frac{1}{N^{k}} N(N-1)\cdots (N-k+1)=1.
	\end{align*}
\end{proof}

\begin{cor}[Tensor Central Limit theorem]\label{cor_tensor_ctl}
	If the independence graph $G_N$ of the variables $a_1,a_2,\ldots,a_N$ is the complete graph,  then $(a_1 + \cdots + a_N)/\sqrt{N}$ converges in moments as $N \to \infty$ to the Gaussian distribution $\frac{1}{\sqrt{2 \pi}}  \exp(-t^2 / 2 )  \, dt$. 
\end{cor}
\begin{proof}
Again, we only consider $m=2k$.  Now, for all $\bm{i}:[m]\to [N]$, $G_{\mathbf{i}}$ is a complete graph. So,  $\mathbf{1}_{G_{\pi(\bm{i})} \subseteq G_{\mathbf{i}}}$ is one for all $\pi \in \mathcal{P}_2(2k)$. Thus,
	\begin{align*}
		\lim_{N\to \infty}\varphi\left( \left(\frac{a_1+\cdots+a_N}{\sqrt{N}} \right)^{2k} \right)&=\lim_{N\to \infty}\frac{1}{N^{k}}\sum\limits_{\pi \in \mathcal{P}_2(2k)} \sum\limits_{\underset{\ker[\bm{i}]=\pi}{\bm{i}:[m]\to [N]}} 1\\
		&=\lim_{N\to \infty}\frac{1}{N^{k}} N(N-1)\cdots (N-k+1) \# \mathcal{P}_2(2k)\\
		&=\# \mathcal{P}_2(2k).
	\end{align*}
\end{proof}

\begin{cor}[Monotone Central Limit theorem]\label{cor_monotone_ctl}
	If the independence graph $G_N$ of the variables $a_1,a_2,\ldots,a_N$ has edge set $E_{N} = \{ (j,i) \in [N]^2 : i < j\}$, then $(a_1 + \cdots + a_N)/\sqrt{N}$ converges in moments as $N \to \infty$ to the arcsine distribution $\frac{1}{2 \pi}  \sqrt{4-t^2} \, dt$. 
\end{cor}
\begin{proof}
	
	Let $\pi\in\mathcal{P}_2(2k)$, and $G_\pi$ be its nesting-crossing graph.
	Note that if $\pi$ has a crossing, $G_{\mathbf{i}}$ cannot be a subgraph of $G_\pi$ for all $\mathbf{i}$, since $G_{\mathbf{i}}$ has no double edges while $G_\pi$ does.

	So let $\pi\in\mathcal{NC}_2(2k)$, with blocks $b_1,\dots,b_k$. The nesting-crossing graph is just a nesting graph. In this case in order that $G_\pi\subset G_{\mathbf{i}}$  we need that $i_l<j_m$ whenever $b_l$ is nested with $b_m$. To count such indices we such a unorder subset of indices of size $k$, and count the number of ordering in the indices which satify this condition.  Thus the cardinality of the set 
	$\{i\in[N]_\pi^k|G_\pi\subset G_{\mathbf{i}}\}$ equals $\binom{N}{k}*M(\pi)$, where $M(\pi)$ is the number of labellings $L:\pi\to[n]$ such $L(b_i)\leq L(b_j)$ if $b_i$ is nested in $b_j$.
	
	In the limit for a partition $\pi\in\mathcal{NC}_2(2k)$ we get that 
	$$ \lim_{N\to \infty}\frac{1}{N^{k}}\sum\limits_{\substack{\bm{i}:[m]\to [N]\\ \ker[\bm{i}]=\pi}} \mathbf{1}_{G_\pi\subseteq G_{\mathbf{i}}}=\lim_{N\to \infty}\frac{1}{N^{k}}\binom{N}{k} M(\pi)=\frac{M(\pi)}{k!}. $$
	Thus summing over all pair partitions we get
	\begin{align*}
		\lim_{N\to \infty}\varphi\left( \left(\frac{a_1+\cdots+a_N}{\sqrt{N}} \right)^{2k} \right)&=\sum\limits_{\pi \in \mathcal{NC}_2(2k)}  \lim_{N\to \infty}\frac{1}{N^{k}}\sum\limits_{\substack{\bm{i}:[m]\to [N]\\ \ker[\bm{i}]=\pi}}  \mathbf{1}_{G_\pi\subseteq G_{\mathbf{i}}}\\&=\sum\limits_{\pi \in \mathcal{NC}_2(2k)}   \frac{M(\pi)}{k!}\\
		&= \frac{1}{k!}.\# \mathcal{M}_2(2k) 
	\end{align*}
 where $\mathcal{M}_2(2k)$ denote the set of monotone pairings, see   \cite{AHLV,LeSa06,HS}.

\end{proof}

Not all sequences of graphs lead to a limiting distribution, even if $G_{n-1}\subset G_n$ for all $n$. We present and example to show this.

\begin{exa}
	Consider the following sequence of graphs:\begin{itemize} \item $G_0=(\{1\},\emptyset)$. \item For all $n$, $G_{2n}=G_{2n-1}\cup \tilde G_{2n-1}$, where $\tilde G_{2n-1}$ is a copy of $G_{2n-1}$ 
		\item   For all $n$, $G_{2n+1}=\hat H_{2n+1}$, $H_{2n+1}=G_{2n}\cup \tilde G_{2n}$,   where $\hat H_{2n+1}$ is the graph $H_{2n+1}$ where we add the edges between the vertices of  $G_{2n}$ and $\tilde G_{2n}$.
		
	\end{itemize}
	The first of these graphs are shown in Figure \ref{counterCLT}.

\begin{figure}
\begin{center}
\begin{tikzpicture}


\node[circle,draw=black, fill=black, minimum size = 0.2cm, inner sep=0pt] (c) at (-1,2.4){};

\node (d) at (-1,-.5){$G_0$};


\node[circle,draw=black, fill=black, minimum size = 0.2cm, inner sep=0pt] (c) at (.5,1){};

\node (d) at (.5,-.5){$G_1$};

\node[circle,draw=black, fill=black, minimum size = 0.2cm, inner sep=0pt] (c) at (.5,4){};

\draw [line width=1pt, black]  (.5,1) -- (.5,4) ;


\node (d) at (2.5,-.5){$G_2$};

\node[circle,draw=black, fill=black, minimum size = 0.2cm, inner sep=0pt] (c) at (2.5,0){};

\node[circle,draw=black, fill=black, minimum size = 0.2cm, inner sep=0pt] (c) at (2.5,2){};

\node[circle,draw=black, fill=black, minimum size = 0.2cm, inner sep=0pt] (c) at (2.5,3){};

\node[circle,draw=black, fill=black, minimum size = 0.2cm, inner sep=0pt] (c) at (2.5,5){};

\draw [line width=1pt, black]  (2.5,0) -- (2.5,2) ;
\draw [line width=1pt, black]  (2.5,3) -- (2.5,5) ;


\node (d) at (5,-.5){$G_3$};

\node[circle,draw=black, fill=black, minimum size = 0.2cm, inner sep=0pt] (c) at (6,0){};
\node[circle,draw=black, fill=black, minimum size = 0.2cm, inner sep=0pt] (c) at (6,2){};
\node[circle,draw=black, fill=black, minimum size = 0.2cm, inner sep=0pt] (c) at (6,3){};
\node[circle,draw=black, fill=black, minimum size = 0.2cm, inner sep=0pt] (c) at (6,5){};

\node[circle,draw=black, fill=black, minimum size = 0.2cm, inner sep=0pt] (c) at (4,0){};
\node[circle,draw=black, fill=black, minimum size = 0.2cm, inner sep=0pt] (c) at (4,2){};
\node[circle,draw=black, fill=black, minimum size = 0.2cm, inner sep=0pt] (c) at (4,3){};
\node[circle,draw=black, fill=black, minimum size = 0.2cm, inner sep=0pt] (c) at (4,5){};

\draw [line width=1pt, black]  (6,0) -- (6,2) ;
\draw [line width=1pt, black]  (6,3) -- (6,5) ;
\draw [line width=1pt, black]  (4,0) -- (4,2) ;
\draw [line width=1pt, black]  (4,3) -- (4,5) ;
\draw [line width=1pt, black]  (6,0) -- (4,2) ;
\draw [line width=1pt, black]  (6,3) -- (4,5) ;
\draw [line width=1pt, black]  (4,0) -- (6,2) ;
\draw [line width=1pt, black]  (4,3) -- (6,5) ;
\draw [line width=1pt, black]  (6,0) -- (4,0) ;
\draw [line width=1pt, black]  (6,2) -- (4,2) ;
\draw [line width=1pt, black]  (6,3) -- (4,3) ;
\draw [line width=1pt, black]  (6,5) -- (4,5) ;
\draw [line width=1pt, black]  (6,0) -- (4,5) ;
\draw [line width=1pt, black]  (6,2) -- (4,3) ;
\draw [line width=1pt, black]  (6,2) -- (4,5) ;
\draw [line width=1pt, black]  (6,0) -- (4,3) ;
\draw [line width=1pt, black]  (6,5) -- (4,2) ;
\draw [line width=1pt, black]  (6,5) -- (4,0) ;
\draw [line width=1pt, black]  (6,3) -- (4,0) ;
\draw [line width=1pt, black]  (6,3) -- (4,2) ;


\node (d) at (9.5,-.5){$G_4$};

\node[circle,draw=black, fill=black, minimum size = 0.2cm, inner sep=0pt] (c) at (7.5,0){};
\node[circle,draw=black, fill=black, minimum size = 0.2cm, inner sep=0pt] (c) at (7.5,2){};
\node[circle,draw=black, fill=black, minimum size = 0.2cm, inner sep=0pt] (c) at (7.5,3){};
\node[circle,draw=black, fill=black, minimum size = 0.2cm, inner sep=0pt] (c) at (7.5,5){};

\node[circle,draw=black, fill=black, minimum size = 0.2cm, inner sep=0pt] (c) at (9,0){};
\node[circle,draw=black, fill=black, minimum size = 0.2cm, inner sep=0pt] (c) at (9,2){};
\node[circle,draw=black, fill=black, minimum size = 0.2cm, inner sep=0pt] (c) at (9,3){};
\node[circle,draw=black, fill=black, minimum size = 0.2cm, inner sep=0pt] (c) at (9,5){};

\node[circle,draw=black, fill=black, minimum size = 0.2cm, inner sep=0pt] (c) at (10,0){};
\node[circle,draw=black, fill=black, minimum size = 0.2cm, inner sep=0pt] (c) at (10,2){};
\node[circle,draw=black, fill=black, minimum size = 0.2cm, inner sep=0pt] (c) at (10,3){};
\node[circle,draw=black, fill=black, minimum size = 0.2cm, inner sep=0pt] (c) at (10,5){};

\node[circle,draw=black, fill=black, minimum size = 0.2cm, inner sep=0pt] (c) at (11.5,0){};
\node[circle,draw=black, fill=black, minimum size = 0.2cm, inner sep=0pt] (c) at (11.5,2){};
\node[circle,draw=black, fill=black, minimum size = 0.2cm, inner sep=0pt] (c) at (11.5,3){};
\node[circle,draw=black, fill=black, minimum size = 0.2cm, inner sep=0pt] (c) at (11.5,5){};

\draw [line width=1pt, black]  (7.5,0) -- (7.5,2) ;
\draw [line width=1pt, black]  (7.5,3) -- (7.5,5) ;
\draw [line width=1pt, black]  (9,0) -- (9,2) ;
\draw [line width=1pt, black]  (9,3) -- (9,5) ;
\draw [line width=1pt, black]  (7.5,0) -- (9,2) ;
\draw [line width=1pt, black]  (7.5,3) -- (9,5) ;
\draw [line width=1pt, black]  (9,0) -- (7.5,2) ;
\draw [line width=1pt, black]  (9,3) -- (7.5,5) ;
\draw [line width=1pt, black]  (7.5,0) -- (9,0) ;
\draw [line width=1pt, black]  (7.5,2) -- (9,2) ;
\draw [line width=1pt, black]  (7.5,3) -- (9,3) ;
\draw [line width=1pt, black]  (7.5,5) -- (9,5) ;
\draw [line width=1pt, black]  (7.5,0) -- (9,5) ;
\draw [line width=1pt, black]  (7.5,2) -- (9,3) ;
\draw [line width=1pt, black]  (7.5,2) -- (9,5) ;
\draw [line width=1pt, black]  (7.5,0) -- (9,3) ;

\draw [line width=1pt, black]  (7.5,5) -- (9,2) ;
\draw [line width=1pt, black]  (7.5,5) -- (9,0) ;

\draw [line width=1pt, black]  (7.5,3) -- (9,0) ;
\draw [line width=1pt, black]  (7.5,3) -- (9,2) ;

\draw [line width=1pt, black]  (10,0) -- (10,2) ;
\draw [line width=1pt, black]  (10,3) -- (10,5) ;
\draw [line width=1pt, black]  (11.5,0) -- (11.5,2) ;
\draw [line width=1pt, black]  (11.5,3) -- (11.5,5) ;
\draw [line width=1pt, black]  (10,0) -- (11.5,2) ;
\draw [line width=1pt, black]  (10,3) -- (11.5,5) ;
\draw [line width=1pt, black]  (11.5,0) -- (10,2) ;
\draw [line width=1pt, black]  (11.5,3) -- (10,5) ;
\draw [line width=1pt, black]  (10,0) -- (11.5,0) ;
\draw [line width=1pt, black]  (10,2) -- (11.5,2) ;
\draw [line width=1pt, black]  (10,3) -- (11.5,3) ;
\draw [line width=1pt, black]  (10,5) -- (11.5,5) ;
\draw [line width=1pt, black]  (10,0) -- (11.5,5) ;
\draw [line width=1pt, black]  (10,2) -- (11.5,3) ;
\draw [line width=1pt, black]  (10,2) -- (11.5,5) ;
\draw [line width=1pt, black]  (10,0) -- (11.5,3) ;

\draw [line width=1pt, black]  (10,5) -- (11.5,2) ;
\draw [line width=1pt, black]  (10,5) -- (11.5,0) ;

\draw [line width=1pt, black]  (10,3) -- (11.5,0) ;
\draw [line width=1pt, black]  (10,3) -- (11.5,2) ;

\end{tikzpicture}

\end{center}
\caption{ \label{counterCLT} The sequence of graphs $G_i$, which do not satisfy a CLT.}
\end{figure}
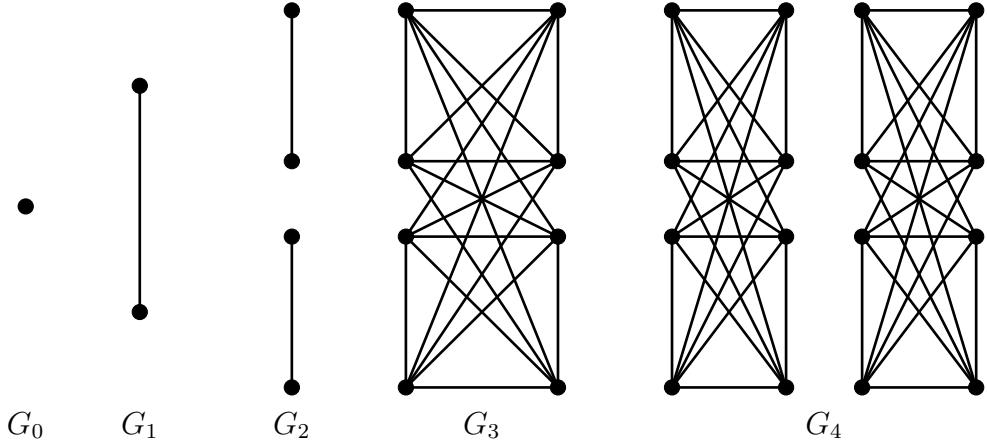

	Now we consider where Bernoulli variables, $(X_i)^\infty_{i=1}$, and suppose that $X_1,\dots,X_{2^n}$ are BMT independent with respect to $G_n$; this is consistent by the construction of $G_n$.
	We are interested in the normalized sum $$Y_n=\frac{X_1+\cdots+X_{2^n}}{2^{n/2}}.$$It is easy to calculate the first three moments
	\begin{eqnarray}
		\phi(Y_n)=0, \quad
		\phi(Y_n^2)=1,\quad
		\phi(Y_n^3)=0.
	\end{eqnarray}
	For the fourth moment we calculate a recursion for $n$ odd or even.
	\begin{eqnarray*}\phi(Y_{2n}^4)&=\frac{1}{4}(\phi((Y_{2n-1})^4)+\phi(({Y}_{2n-1})^4)+2\phi((Y_{2n-1})^2))=\frac{1}{2}(\phi((Y_{2n-1})^4)+1),\\
		\phi(Y_{2n+1}^4)&=\frac{1}{4}(\phi((Y_{n})^4)+\phi(({Y}_{n-1})^4)+6\phi((Y_{n})^2))=\frac{1}{2}(\phi((Y_{2n})^4)+3).
	\end{eqnarray*}
	From where, $\phi(Y_{2n+2}^4)=\frac{1}{4}\phi(Y_{2n}^4)+\frac{5}{4}$
	and
	$\phi(Y_{2n+2}^4)=\frac{1}{4}\phi(Y_{2n}^4)+\frac{7}{4}.$
	Hence, one sees that $$\phi(Y_{2n}^4)\to 5/3\text{ and }\phi(Y_{2n+1}^4)\to 7/3.$$

 Thus we do not have convergence in moments and neither convergence in distribution, by Remark \ref{oddmoments}.

\end{exa}

\subsection{Further properties of CLT}

We have observed some properties of BMT central limit theorem such as determinacy of moments and symmetry.  Here, we want to consider also some properties in relation with the associated independence graphs. Firstly, we show that if two graphs differ by a small number of edges then they provide the same limiting distribution for the CLT. Secondly, we consider the boundedness of support in the limiting distribution and finally, we provide a description and examples on how consistency is reflected into Boolean, tensor or monotone convolutions.

\begin{thm}[Perturbation]\label{thm_perturbation}
	Let $a_1,\ldots,a_N,b_1,\ldots,b_N$ be centered variables with unit variance and uniformly bounded moments of all orders. 
	Suppose $G_N$ and $H_N$ are the independence graphs of $a_1,\ldots,a_N$ and $b_1,\ldots,b_N$, respectively. 
	If the symmetric difference of $G_N$ and $H_N$ has order $o(N^2)$, then 
	\[
	\lim_{N \to \infty} \left[ \varphi	\left( \frac{a_1 + \cdots + a_N}{\sqrt{N}} \right)^m - 
	\varphi	\left( \frac{b_1 + \cdots + b_N}{\sqrt{N}} \right)^m  \right]
	=  0 
	\qquad \forall \ m \geq 1.
	\]
	%
	Consequently, $(a_1 + \cdots + a_N)/\sqrt{N}$ has a limiting distribution if and only if $(b_1 + \cdots + b_N)/\sqrt{N}$ does; moreover, the two limiting distributions coincide if any of them exists. 
\end{thm}
\begin{proof}
	First note that $G_{\pi(\bm{i})} = H_{\pi(\bm{i})}$ since this is a graph that depends only on the partition $\pi \in P(m)$ and a tuple $\bm{i}=(i_1,\ldots,i_m)$ with $\ker[\bm{i}] = \pi$ and not on $G_N$ nor $H_N$.
	So, to avoid confusion, let us put $F_{\pi(\bm{i})}= G_{\pi(\bm{i})}= H_{\pi(\bm{i})}$. 
	Due to Theorem \ref{thm_bmt_clt}, it is enough to show that 
	\[
	0 = \lim_{N \to \infty}
	\sum_{\substack{ {\bm i}:[m]\rightarrow [N] \\ \ker[{\bm i}]=\pi}} 
	N^{-m/2} \left( 
	\bm{1}_{F_{\pi(\bm{i})}  \subseteq G_{\bm i}}   - 
	\bm{1}_{F_{\pi(\bm{i})}  \subseteq H_{\bm i}}  \right)
	\]
	for each pairing partition $\pi \in P_2(m)$. 
	Since $\bm{1}_{F_{\pi(\bm{i})}  \subseteq G_{\bm i}} -	\bm{1}_{F_{\pi(\bm{i})}  \subseteq H_{\bm i}} = 0$ unless $F_{\pi(\bm{i})}  \subseteq G_{\bm i}  $ and $F_{\pi(\bm{i}) } \not\subseteq H_{\bm i}  $ or  $F_{\pi(\bm{i})}  \subseteq G_{\bm i} $ and $F_{\pi(\bm{i}) } \not\subseteq H_{\bm i}$, 
	a condition that holds only if $G_{\bm i} \not\subseteq H_{\bm i}$ or $H_{\bm i} \not\subseteq G_{\bm i}$, we obtain 
	\[
	\sum_{\substack{ {\bm i}:[m]\rightarrow [N] \\ \ker[{\bm i}]=\pi}} 
	\abs{\bm{1}_{F_{\pi(\bm{i})}  \subseteq G_{\bm i}} -	\bm{1}_{F_{\pi(\bm{i})}  \subseteq H_{\bm i}}} 
	\leq 
	\abs{S_{N,\pi}} 
	\]
	where ${S}_{N,\pi} =
	\left\{ \bm{i}:[m]\rightarrow[N] : \ker[\bm{i}] = \pi \text{ and } G_{\bm i} \not\subseteq H_{\bm i} \text{ or } H_{\bm i} \not\subseteq G_{\bm i} \right\}.$

	Take $\pi = \{ V_1, V_2, \ldots, V_r\} \in P_2(m)$ with $r$ the number of blocks of $\pi$. 
	For a given tuple $\bm{i} = (i_1,\ldots,i_m)$, the graphs $G_{\bm{i}}$ and $H_{\bm{i}}$ have the same vertex set $\{i_k : k =1,2,\ldots,m\}$. 
	And hence,  $\ker[\bm{i}] = \pi$ implies $\bm{i} \in S_{N,\pi}$ only if $G_{\bm{i}}$ has at least one edge that $H_{\bm{i}}$ does not or $H_{\bm{i}}$ has at least one edge that $G_{\bm{i}}$ does not. 
	Thus, for each $\bm{i} \in $, at least one of the graphs $G_{\bm{i}}$ and $H_{\bm{i}}$ shares at least one edge with $G_N \triangle H_N$ the symmetric difference of $G_N$ and $H_N$. 
	This implies all elements in ${S}_{N,\pi}$ can be constructed in the following way: (1) pick two distinct blocks $V,W \in \pi$, an edge $(j_V,j_W) \in E(G_N \triangle H_N)$, and $r-2$ distinct values $j_U\in [N] \setminus \{ j_V, j_W\}$ for $U \in \pi \setminus \{V,W\}$; (2) take $\bm{i} = (i_1,\ldots,i_m)$ where $i_k = j_U$ provided $k \in U$. 
	The previous construction is not necessarily injective but it does define a set containing ${S}_{N,\pi}$.  
	So, we get  
	\[
	\abs{ S_{N,\pi} }
	\leq 
	r (r-1) \cdot \abs{ E(G_N \triangle H_N) } \cdot  (N-2) (N-3) \cdots (N-(r-1))
	\leq 
	r (r-1) \cdot \abs{ E(G_N \triangle H_N) } \cdot  N^{r-2}
	\]
	The desired result then follows since  $r \leq m/2$ and $\lim_{N \to \infty}\abs{ E(G_N \triangle H_N) }/N^2 = 0$ imply 
	\[ 
	\lim_{N \to \infty}  \frac{ r (r-1) \cdot \abs{ E(G_N \triangle H_N) } \cdot  N^{r-2} }{N^{m/2}}  = 0 . 
	\]
\end{proof}

\begin{cor}
	If the independence graph $G_N$ of the variables $a_1,a_2,\ldots,a_N$ has order $o(N^2)$, then  $(a_1 + \cdots + a_N)/\sqrt{N}$ converges in moments as $N \to \infty$ to the Bernoulli distribution $(\delta_{-1} + \delta_{+1})/2$. 
\end{cor}

\begin{cor}
	If the independence graph $G_N$ of the variables $a_1,a_2,\ldots,a_N$ contains no copy of a complete bipartite graph, then $(a_1 + \cdots + a_N)/\sqrt{N}$ converges in moments as $N \to \infty$ to the Bernoulli distribution $(\delta_{-1} + \delta_{+1})/2$. 
\end{cor}
\begin{proof}
	This follows from Theorem \ref{thm_perturbation} and the Kovari-Sós-Turán theorem, the latter stating that for fixed integers $s \geq r \geq 2$ any $N$-vertex graph with at least $CN^{2-\frac{1}{r}}$ edges contains a complete bipartite subgraph $K_{r,s}$.  
\end{proof}

\begin{prop}

	Let for each $N$, let $a_1,\ldots,a_N,$ and $b_1,\ldots,b_N,$ be a couple of tuples of i.d. selfadjoint random variables  with mean zero and variance one. 
	Suppose that $G_N$ and $H_N$ are the independence graphs of $a_1,\ldots,a_N$ and $b_1,\ldots,b_N$, respectively. 
	If $G_N$ is a subgraph of $J_N$, for all $N$ and if $(a_1 + \cdots + a_N)/\sqrt{N}$ and $(b_1 + \cdots + b_N)/\sqrt{N}$ converge in moments to $X$ and $Y$, respectively, then $\|X\|_\infty \leq \|Y\|_\infty$.
\end{prop}

\begin{proof}

  If $Y$ has unbounded support, then the statement follows by definition. Otherwise $\|Y\|_\infty$ is finite and thus determined by moments.

  Notice that both  limiting distributions do not depend on the choice of $a_i$'s, and we may assume that $a_i$ has the same distribution as $b_i$, for all $i$.
  
 Now, for any index $\mathbf{i}=(i_1,\dots,i_k)$,
 
	\[ \varphi(a_{i_1}\cdots a_{i_k})
	\ \, \leq 
	\varphi(b_{i_1}\cdots a_{i_k}) \]
	Consequently, the moments of 
$(b_1 + \cdots + b_N)/\sqrt{N}$ are larger than the moments of $(a_1 + \cdots + a_N)/\sqrt{N}$. 
By taking limits we see that, if we denote the even moments of $X$ by $(m_{2n}(X))_{n>0}$,  and the moments of Y by $(m_{2n}(Y))_n>0$, we have the inequality
 $$m_{2n}(X)\leq m_{2n}(Y),\quad \text{for all } n.$$
 This means that $(m_{2n}(X))^{1/n}\leq \|Y\|_\infty $ and thus, letting $n\to \infty$ we see that  $\|X\|_\infty \leq  \|Y\|_\infty$.
\end{proof}

\begin{thm}
	Let $M_N$ denote the greatest integer $M \geq 0$ so that the full graph $K_M$ is contained if $G_N$. 
	If $\liminf_{N \to \infty} M_N/N >0$, then the limit distribution of $(a_1 + \cdots + a_N)/\sqrt{N}$ if it exists has non-compact support. 
\end{thm}
\begin{proof}
	Suppose $(a_1 + \cdots + a_N)/\sqrt{N}$ has an analytical limiting distribution, i.e., there exists a probability measure $\mu$ on the real line so that 
	\[
	\lim_{N \to \infty} \varphi \left(\frac{a_1 + \cdots + a_N}{\sqrt{N}}\right)^{k}
	= \ \int_{-\infty}^{+\infty} t^k  \, d \mu(t)
	=: m_k(\mu)
	\]
	for all integers $k \geq 1$. 
	Note that if $\mu$ has compact support, so $\mu((-\infty,-L)\cup (+L,+\infty)) = 0$ for some $L \geq 0$, then $\sup_{k \geq 1} \sqrt[k]{\abs{m_k(\mu)}} \leq L < \infty$. 
	Thus, to prove $\mu$ has no compact support, it enough to show that $\sup_{k \geq 1} \sqrt[2k]{\abs{m_{2k}(\mu)}} = + \infty $. 

	By hypothesis, for each integer $N \geq 1$, there exists a set $J_N = \{ j_1 < j_2< \cdots < j_{M_N}\} \subset[N] $ so that $G_N$ contains the full graph on $J_N$, so  $(j_k,j_\ell)$ is an edge of $G_N$ for any distinct $j_k,j_\ell \in J_N$. 
	Thus, if $\bm{i} = (i_1,i_2,\ldots,i_{2k})$ satisfies $i_r \in J_N$ for $r=1,2,\ldots,2k$, the graph $G_{\bm i}$ is the full graph since it is the restriction of $G_N$ to vertex set $\{i_r \mid r = 1, 2, \ldots, 2k \} \subset J_N$, and hence $G_{\pi(\bm{i})}$ is always a subgraph of $G_{\bm i}$ with $\pi = \ker[\bm{i}]$. 
	So, for any pairing partition $\pi \in P_2(2k)$, we get   
	\[
	\sum_{\substack{ \bm{i}:[2k]\rightarrow [N] \\ \ker[{\bm i}]=\pi}}  
	\bm{1}_{G_{\pi(\bm{i})}  \subseteq G_{\bm i}}  
	\ \ \geq 
	\sum_{\substack{ \bm{i}:[2k]\rightarrow J_N \\ \ker[{\bm i}]=\pi}}  
	\bm{1}_{G_{\pi(\bm{i})}  \subseteq G_{\bm i}}
	= \ 
	M_N (M_N -1 ) \cdots (M_N-k+1) 
	\]
	where $k$ is the number of blocks in $\pi$ since it is a pairing. 
	Put $2 C = \liminf_{N \to \infty} M_N /N >0$. 
	Thus, for each $k \geq 1$, there exists $\tilde{N}_k \geq 1$ so that $N \geq \tilde{N}_k$ implies $M_N \geq k$ and 
	\[
	\left( \frac{M_N}{N} \right)  \left( \frac{M_N-1}{N} \right) \cdots \left( \frac{M_N-k+1}{N} \right)
	\geq 
	C^k.   
	\]
	Hence, Theorem \ref{thm_bmt_clt} implies $m_{2k}(\mu) \geq \abs{ P_2(2k) } C^k$ since the last two inequalities give 
	\[
	\sum_{\pi \in P_2(2k)}  \ \ 
	N^{-k}
	\sum_{\substack{ \bm{i}:[2k]\rightarrow [N] \\ \ker[{\bm i}]=\pi}}  
	\bm{1}_{G_{\pi(\bm{i})}  \subseteq G_{\bm i}}  
	\ \ \geq 
	\sum_{\pi \in P_2(2k)}  \ \ 
	C^k  
	\ \ =  \ \ 
	\abs{ P_2(2k) } C^k
	\]
	for $N \geq \tilde{N}_k$. 
	Finally, $\abs{ P_2(2k) } = (2k-1)(2k-3) \cdots (1)$ implies $\sup_{k \geq 1} \sqrt[2k]{\abs{m_{2k}(\mu)}} = + \infty$. 
\end{proof}

Finally, we consider the role of consistency in the BMT Central Limit Theorems.

\begin{prop}\label{prop_disjoint_conv}
	Let $a_1,\ldots,a_{M_N},a_{M_N+1},\ldots,a_{M_N+L_N}$ be centered variables with unit variance, uniformly bounded moments of all orders, and independence graph $G_{M_N + L_N}$ where $N = M_N +L_N$.  
	Suppose $(a_1 + \cdots + a_{M_N})/\sqrt{M_N}$ and $(a_{M_N+1} + \cdots + a_{M_N+L_N})/\sqrt{L_N}$ converge in moments to $a_M$ and $a_L$, respectively. 
	If $t = \lim_{N \to \infty} M_N / N $ exists, then $(a_1 + \cdots + a_N)/\sqrt{N}$ converges in moments to $\sqrt{t} \cdot a_M \, + \, \sqrt{1-t} \cdot a_L$ with $a_M$ and $a_L$ 
	\begin{enumerate}[(1)]
		\item boolean independent if $(i,j),(j,i) \notin E_{M_N + L_N}$ whenever $i \in [M_N]$ and $j \in [M_N + L_N] \setminus [M_N]$
		\item monotone independent if $(j,i) \in E_{M_N + L_N}$ and $(i,j) \notin E_{M_N + L_N}$ whenever $i \in [M_N]$ and $j \in [M_N + L_N] \setminus [M_N]$
		\item tensor independent if $(i,j),(j,i) \in E_{M_N + L_N}$ whenever $i \in [M_N]$ and $j \in [M_N + L_N] \setminus [M_N]$
	\end{enumerate}
\end{prop}

\begin{proof}
	Take $a_{N,1} = (a_1 + \cdots + a_{M_N})/\sqrt{N}$ and $a_{N,2} = (a_{M_N+1} + \cdots + a_{M_N+L_N})/\sqrt{N}$. 
	Thus, for each integer $m\geq 1$, we get   
	\begin{align*}
		\varphi \left(\frac{a_1 + \cdots + a_N}{\sqrt{N}}\right)^m
		=
		\sum_{i_1,\ldots,i_m=1}^2  \varphi \, ( a_{N,i_1} a_{N,i_2} \cdots  a_{N,i_m} ) . 
	\end{align*}
	Now, if \emph{(1) ((2),(3))} holds, we have $a_{N,1}$ and $a_{N,2}$ boolean (monotone, tensor, respectively) independent due to Proposition \ref{prop_associativity}, and hence the last equality implies 
	\[ 
	\lim_{N \to \infty} \varphi
	\left(\frac{a_1 + \cdots + a_N}{\sqrt{N}}\right)^m
	=
	\varphi \, (\sqrt{t} \cdot a_M \, + \, \sqrt{1-t} \cdot a_L)^m
	\]
	with $a_M$ and $a_L$ boolean (monotone, tensor, respectively) independent since $a_{N,1}$ and $a_{N,2}$ converge in moments to $\sqrt{t} \cdot a_M$ and $\sqrt{1-t} \cdot a_L$, respectively. 
\end{proof}

\begin{exa}
	Suppose the independence graph $G_N$ of the variables $a_1,\cdots,a_N$ is the the complete bipartite graph $K_{M_N,L_N}$ with $N = M_N + L_N$. 
	From Corollary \ref{cor_boolean_ctl}, it follows that each $(a_1 + \cdots + a_{M_N})/\sqrt{M_N}$ and $(a_{M_N+1} + \cdots + a_{M_N+L_N})/\sqrt{L_N}$ converge in moments to a Bernoulli distribution $(\delta_{-1} + \delta_{+1})/2$. 
	If additionally we have $\lim_{N\to \infty} = t$ exists, then $0 \leq t \leq 1$ and $(a_1+\cdots+a_N)/\sqrt{N}$ converge in moments to 
	\[
	\left[
	\frac{1}{2}		\delta_{-\sqrt{t}}
	+
	\frac{1}{2}		\delta_{+\sqrt{t}}
	\right]  \ast \left[
	\frac{1}{2}		\delta_{-\sqrt{1-t}}
	+
	\frac{1}{2}		\delta_{+\sqrt{1-t}}
	\right] 
	\]
	where $\ast$ denotes the classical convolution of measures. 
\end{exa}

\begin{exa}
	Suppose the independence graph $G_N$ of the variables $a_1,\cdots,a_N$ is the the disjoint union of two complete graphs $K_{M_N}$ and $K_{L_N}$ with vertex sets $[M_N]$ and $[M_N+L_N]\setminus[M_N]$ ,respectively, and $N = M_N + L_N$. 
	From Corollary \ref{cor_tensor_ctl}, it follows that each $(a_1 + \cdots + a_{M_N})/\sqrt{M_N}$ and $(a_{M_N+1} + \cdots + a_{M_N+L_N})/\sqrt{L_N}$ converge in moments to a normal distribution $\mathcal{N}(0,1)$. 
	Thus, if $\lim_{N\to \infty} = t$ exists, then $0 \leq t \leq 1$ and $(a_1+\cdots+a_N)/\sqrt{N}$ converge in moments to 
	\[
	\mathcal{N}(0,\sqrt{t})
	\, \uplus \,
	\mathcal{N}(0,\sqrt{1-t})
	\]
	where $\uplus$ denotes the Boolean convolution of measures. 
\end{exa}

\begin{exa}
	Suppose $G_N$ is the Turán Graph $T(N,r_N)$, a complete multi-partite graph formed by partitioning a set of $N$ vertices into $r_N$ subsets with sizes as equal as possible. 
	This graph becomes the empty graph (when $r_N=1$), the complete graph (when $r_N=N$), and a complete bipartite graph (when $r_N=2$). 
	If $r_N \to r <\infty$, when $N\to \infty$, then the bmt central limit theorem associated to $T(N,r_N)$ has as limit distribution a convolution of $r$ Bernoulli random variables. 
\end{exa}

\begin{rem}
	Proposition \ref{prop_disjoint_conv} and last examples shows the number of edges does not determine the support compactness of the limit distribution in Theorem \ref{thm_bmt_clt}. 
	Indeed, suppose $G_N=K_{\lfloor N/2 \rfloor} \sqcup \emptyset_{\lceil N/2 \rceil}$ and $H_N=K_{\lfloor N/2 \rfloor,\lceil N/2 \rceil}$ and let $n_N$ and $m_N$ denote the the number of edges of $G_N$ and  $H_N$, respectively. 
	Then, we have 
	$$n_N=\frac{\lfloor \frac{N}{2} \rfloor (\lfloor \frac{N}{2} \rfloor-1) }{2}\leq \lfloor \frac{N^2}{8} \rfloor\leq \lfloor \frac{N^2}{4} \rfloor\leq \lfloor \frac{N}{2} \rfloor\lceil \frac{N}{2} \rceil=m_N.$$
	Thus, the limit distribution associated to $H_N$ is compactly supported (being the classical convolution of two Bernoulli distributions), while the limit distribution associated to $G_N$ is not (being the boolean convolution of a Gaussian distribution and a Bernoulli distribution).
\end{rem}

\section{Poisson Limit Theorem}

Apart from the central limit theorem,  the law of rare events also known as Poisson limit theorem is probably the the most important theorem in probability. 

In this limit theorem,  one considers sums of independent Bernoulli variables with common parameter $\lambda/n$ and considers its limit.  In BMT independence we can have such limits.  The combinatorics appearing are similar as the one for the central limit theorem. Not surprisingly, the main difference is that here we need to consider the set of all instead of  pair partitions to all partitions.

\begin{thm}\label{thm_poisson}
	Suppose $a_1, a_2, \ldots, a_N$ BMT independent with respect to the graph $G_N$. If each $a_i$ has Bernoulli distribution $\mu_N = (1-\frac{\lambda}{N})\delta_0 + \frac{\lambda}{N} \delta_{1} $, then the moments of the sum $a_1+\cdots a_N$ satisfy
	\begin{eqnarray*}
		\varphi\left[ \left( a_1+\cdots+a_N  \right)_{}^m\right]
		\ \, =
		\sum_{\pi \in P(m)}  
		\lambda^{\#(\pi)}  \ N^{-\#(\pi)}
		\sum_{\substack{ {\bm i}:[m]\rightarrow [N] \\ \ker[{\bm i}]=\pi}}  
		\bm{1}_{G_{\pi(\bm{i})}  \subseteq G_{\bm i}}  
		\quad	+ \quad O(N^{-1}) \ . 
	\end{eqnarray*} 
\end{thm}
\begin{proof}
	%
	%
	Expanding the product $(a_1 + \cdots + a_N)^m$ and using the fact that  $\ker[\bm{i}] = \ker[\bm{i'}]$ defines an equivalence relation on $\{ \bm{i} \mid \bm{i}: [m]\to [N]\}$, we have
	\[
	\varphi	( {a_1 + \cdots + a_N} )^m
	\ \ = 
	\sum_{\pi \in P(m)}  \ 
	\sum_{\substack{ {\bm i}:[m]\rightarrow [N] \\ \ker[{\bm i}]=\pi}}  
	\varphi( a_{i_1} a_{i_2} \cdots a_{i_m} ).
	\]
	Notice that $\varphi(a^k_i) =\lambda/N$ since each $a_i$ has Bernoulli distribution $\mu_N = (1-\frac{\lambda}{N})\delta_0 + \frac{\lambda}{N} \delta_{1}$. 
	So, for each partition $\pi \in P(m)$, the BMT independence of $a_1,a_2,\ldots,a_N$  gives  
	\[
	\sum_{\substack{ {\bm i}:[m]\rightarrow [N] \\ \ker[{\bm i}]=\pi}}  
	\varphi( a_{i_1} a_{i_2} \cdots a_{i_m} )
	\ \ = \sum_{\substack{ {\bm i}:[m]\rightarrow [N] \\ \ker[{\bm i}]=\pi}}  	\prod_{V \in \ker_G[\bm{i}]} \varphi((a_{i_k})|V)
	\ \ =
	\sum_{\substack{ {\bm i}:[m]\rightarrow [N] \\ \ker[{\bm i}]=\pi}}  
	\left(\frac{\lambda}{N}\right)^{\# (\ker_G[\bm{i}])}			
	\]
	where $\# (\ker_G[\bm{i}])$ denotes the number of blocks in the partition $\ker_G[\bm{i}]$. 
	Now, recall that $\ker_G[\bm{i}]$ is a refinement of $\ker[{\bm i}]$, so we can write    
	\[
	\sum_{\substack{ {\bm i}:[m]\rightarrow [N] \\ \ker[{\bm i}]=\pi}}  
	\left(\frac{\lambda}{N}\right)^{\# (\ker_G[\bm{i}])}
	\ \ = \ \ \  
	\sum_{\substack{ \theta \in P(m) \\ \theta \leq \pi } }
	\sum_{\substack{ {\bm i}:[m]\rightarrow [N] \\ \ker[\bm{i}]=\pi \\ \ker_G[\bm{i}] = \theta}} 
	\left(\frac{\lambda}{N}\right)^{\# (\ker_G[\bm{i}])} .
	\]
	Moreover, for any partitions $\pi, \theta \in P(m)$ with $\theta \leq \pi$, we have
	\[
	\abs{
		\sum_{\substack{ {\bm i}:[m]\rightarrow [N] \\ \ker[\bm{i}]=\pi, \ker_G[\bm{i}] = \theta}} 
		\left(\frac{\lambda}{N}\right)^{\# (\ker_G[\bm{i}])}}
	\leq 
	\left(
	\frac{ \abs{\lambda} }{N} \right)^{\#(\theta) }
	\abs{
		\sum_{\substack{ {\bm i}:[m]\rightarrow [N] \\ \ker[\bm{i}]=\pi }} 
		1 
	}
	\leq 
	\abs{\lambda}^{\#(\theta)}  N^{\#(\pi) - \#(\theta)}
	\]
	where $\#(\pi)$ and $\#(\theta)$ denote the number of blocks in $\pi$ and $\theta$, respectively. 
	Hence, since $\theta \leq \pi$ implies $\#(\theta) \geq \#(\pi)$ with equality only if $\theta = \pi$, we obtain 
	\[
	\sum_{\substack{ {\bm i}:[m]\rightarrow [N] \\ \ker[{\bm i}]=\pi}}  
	\left(\frac{\lambda}{N}\right)^{\# (\ker_G[\bm{i}])}
	\ \ = \ \ 
	\sum_{\substack{ {\bm i}:[m]\rightarrow [N] \\ \ker[\bm{i}]=\pi, \ker_G[\bm{i}] = \pi}} 
	\left(\frac{\lambda}{N}\right)^{\# (\pi)} 
	\quad + \quad O(N^{-1}) .
	\]
	Therefore, due the equivalence of the conditions $G_{\pi(\bm{i})}  \subseteq G_{\bm i}$ and $\ker[\bm{i}]= \ker_G[\bm{i}] $ for any ${\bm i}:[m]\rightarrow [N]$ with $\ker[{\bm i}]=\pi$, we get 
	\[ \varphi\left[ \left( a_1+\cdots+a_N  \right)_{}^m\right]
	\ \, =
	\sum_{\pi \in P(m)}  
	\lambda^{\#(\pi)}  \ N^{-\#(\pi)}
	\sum_{\substack{ {\bm i}:[m]\rightarrow [N] \\ \ker[{\bm i}]=\pi}}  
	\bm{1}_{G_{\pi(\bm{i})}  \subseteq G_{\bm i}}  
	\quad	+ \quad O(N^{-1}) \ .
	\]
\end{proof}


Similarly as for the CLT we may deduce the Boolean, monotone and tensor Poisson limit Theorems, as applications of the above theorem by using similar arguments as in the proofs of Corollaries \ref{cor_boolean_ctl}, \ref{cor_tensor_ctl}  and  \ref{cor_monotone_ctl}.
We leave the details of the proof to the reader. 

\begin{cor}
	Let for all $N$, $a_1,\ldots, a_N$ have a Bernoulli distribution $\mu_N=(1-\frac{\lambda}{N})\delta_0+\frac{\lambda}{N}\delta_1$.  Then as $N\to \infty$
	
	\begin{enumerate} 
		\item If $a_1,\ldots, a_N$ are BMT independent with respect to the complete graph $G_N=K_N$, 
		then the variable $a_1+\cdots a_N$ converges in moments to a classical Poisson distribution.
		\item 
		If $a_1,\ldots, a_N$ are BMT independent with respect to the empty graph $G_N=\emptyset_N$,  then the variable $a_1+\cdots a_N$ converges in moments to a Boolean Poisson distribution, \cite{BoolWo}.
		\item If $a_1,\ldots, a_N$ are BMT independent with respect to the graph $G_N$ associated with a total order $<$, then the variable $a_1+\cdots a_N$ converges in moments to a monotone Poisson distribution, \cite{Mur2000}
	\end{enumerate}
\end{cor}

\section{Concluding Remarks}

We have started the development of new framework both analytical and algebraic that enables us to investigate arbitrary mixtures of Boolean, monotone and tensor independence. 
This setting provides a unified approach for these three notions of independence through a digraph depicting pair-wise independence relations and what we have called the kernel of function subordinated to a digraph. 
Nonetheless, besides establishing Central and Poisson-Type Limit Theorems, there are many interesting problems that remain open. 
We mention a few in this last section.

\begin{enumerate}[(1)]
\item The ideas in this paper and those from \cite{JW2020} could be combined to obtain a corresponding notion of BMFT independence. 
In \cite{JW2020}, the authors introduced the notion of $\mathcal{T}$\emph{-tree} independence, which allowed them to study mixtures of Boolean, free, and monotone independence. 
A closer look at our notion of the kernel subordinated to a digraph reveals that this object splits mixed moments according to the commutation relations from tensor independence, much in the spirit of the reduced non-crossing partitions from mixtures of free and tensor independence in \cite{eps2016,Lambda2004}.  
A notion of  BMFT independence seems likely to arise from combining the maximal non-crossing partitions from $\mathcal{T}$\emph{-tree} independence and the commutation relations encoded in the kernel subordinated to a digraph. 

\item  The CLT and Poisson Limit theorem hint on the possible set of partitions that describe the combinatorics of BMT independence. It would be interesting if such intuition can be further studied in a systematical way. One possible direction and an important open question is to define cumulants with respect to BMT independence.

 	\item It has been shown that if we consider sequences of random variables where  commuting and anticommuting  relations are taken randomly, one obtains a deterministic limit in the CLT, which correspond to $q$-Gaussian distributions \cite{speicher}, see also \cite{blitvic} for a generalization. It would be interesting if considering random directed graphs as the independence graphs in the BMT-central limit theorem one would get some interesting interpolation between, Gaussian, arcsine and symmetric Bernoulli distributions.

\item It would be desirable to classify all probability measures that arise as limits of the CLT for BMT independent random variables.   
	In particular, one should determine the digraphs that recover and generalize the limiting measures obtained in \cite{BM2010} for BM independence.  
	These latter measures were shown to be symmetric and compactly supported, they include the semi-circular law and their even moments satisfy the generalized recurrence relation for Catalan numbers, namely, 
	\[ g_n = \sum_{r=1}^{n} \gamma_r \, g_{r-1}g_{n-r} . \]
	Section 5 in this paper shows that non-compact measures appear in the CLT for BMT independent random variables, so they are not covered by BM independence.

\item  Diving into the last remark, one would like to find out properties of digraphs $G_N=(V_N,E_N)$ that determine the compactness of limiting measures in the CLT for BMT independent random variables. 
This seems to be a non-trivial problem since it amounts to analyze the growth as $N$ goes to infinity of the number of sub-graphs of $G_N$ that are isomorphic to each nesting-crossing graph $G_\pi$. 
Moreover, the edge set $E_N$ must be of order $N^2$ ---i.e., $\lim_{N \to \infty} |E_N|/N^2 > 0$--- to obtain a limiting measure different from the Bernoulli distribution, but at this order both compact and non-compact measure appear. 

\end{enumerate}


\end{document}